\documentclass[12pt,reqno,english]{amsart}
\usepackage{amsmath,amsthm,amsfonts,amssymb,amscd}
\usepackage[latin1]{inputenc}
\usepackage[T1]{fontenc}
\usepackage[all]{xy}
\usepackage{psfrag}
\usepackage{epsfig}
\usepackage{a4wide}
\theoremstyle{plain}
\newtheorem{theorem}{Theorem}
\newtheorem{proposition}{Proposition}
\newtheorem{lemma}{Lemma}

\newtheorem{remark}{Remark}
\newtheorem{corollary}{Corollary}
\newtheorem{definition}{Definition}

\numberwithin{equation}{section}
\newcommand{\op}{\operatorname}

\def \fh {{\mathfrak h}}
\def \fa {{\mathfrak a}}
\def \fN {{\mathfrak N}}
\def \QQ {{\mathbb Q}}
\def \RR {{\mathbb R}}
\def \ZZ {{\mathbb Z}}
\def \NN {{\mathbb N}}
\def\CC {{\mathbb C}}

\def \TT {{\mathbb T}}
\def \SS {{\mathbb S}}
\def \cA {{\mathcal A}}
\def \cC {{\mathcal C}}

\def \cO {{\mathcal O}}

\def \cK {{\mathcal K}}
\def \cH {{\mathcal H}}
\def \cG {{\mathcal G}}

\def \cU {{\mathcal U}}
\def \cR {{\mathcal R}}
\def \cT {{\mathcal T}}
\def \cZ {{\mathcal Z}}
\def \cL {{\mathcal L}}

\def \cN {{\mathcal N}}
\def \cS {{\mathcal S}}
\def \cU {{\mathcal U}}

\def\-{{\setminus}}

\def\Hei{{\rm Heis}}

\begin{document}

\title[Algebraic Anosov actions of nilpotent Lie groups]
      {Algebraic Anosov actions of nilpotent Lie groups}

\author{Thierry Barbot}
\author{Carlos Maquera}
\thanks{The authors would like to thank FAPESP for the partial financial
support. Grants  2009/06328-2, 2009/13882-6 and 2008/02841-4 }

\keywords{Algebraic Anosov action, Cartan subalgebra}

\subjclass[2010]{Primary: 37D40; secondary: 37C85, 22AXX, 22F30}

\date{\today}
%\date{September 1, 2002}

\address{Thierry Barbot\\ Universit\'e d'Avignon et des pays de Vaucluse\\
LANLG, Facult\'e des Sciences\\
33 rue Louis Pasteur\\
84000 Avignon, France.}

\email{thierry.barbot@univ-avignon.fr}

\address{Carlos Maquera\\ Universidade de
S{\~a}o Paulo - S{\~a}o Carlos \\Instituto de ci{\^e}ncias
matem{\'a}ticas e de Computa\c{c}{\~a}o\\
Av. do Trabalhador S{\~a}o-Carlense 400 \\
13560-970 S{\~a}o Carlos, SP\\
Brazil}

 \email{cmaquera@icmc.usp.br}

 \begin{abstract}
 In this paper we classify algebraic Anosov actions of nilpotent Lie groups on closed manifolds, extending the previous results by P. Tomter (\cite{tomter1, tomter2}).
 We show that they are all nil-suspensions over either suspensions of Anosov actions of $\ZZ^k$ on nilmanifolds, or (modified) Weyl chamber actions. We check 
 the validity of the generalized Verjovsky conjecture in this algebraic context. We also point out an intimate relation between algebraic Anosov actions and Cartan 
 subalgebras in general real Lie groups.
 \end{abstract}

 \maketitle

 \medskip

 \thispagestyle{empty}

%%%%%%%%%%%%%%%%%%%%%%%%%%%%%%%%%%%%%%%%
%%%%%%%%%%%%%%%%%%%%%%%%%%%%%%%%%%%%%%%%
\tableofcontents

%%%%%%%%%%%%
 \section*{Introduction}
 %%%%%%%%%%%%

 A locally free action $\phi$ of a group $H$ on a closed manifold $M$ is said to be \textit{Anosov} if there exists  $a\in H$ such that
 $g:=\phi(a,\cdot)$ is normally hyperbolically with respect to the orbit
 foliation. As for Anosov flows, there exists a continuous $Dg$-invariant  splitting of the tangent
 bundle
  $$
  TM=E^{ss}\oplus T\cO \oplus E^{uu}
  $$ such that
  $$
  \begin{array}{cc}
   \|Dg^n|_{E^{ss}}\|\leq Ce^{-\lambda n} & \forall n>0 \\
   \|Dg^n|_{E^{uu}}\|\leq Ce^{\lambda n} & \forall n<0
  \end{array},
  $$
  where $T\cO$ denotes  the $k$-dimensional subbundle of $TM$ that is tangent to the orbits of
  $\phi.$

The most natural examples of Anosov actions  come from algebra: let $G$ be a connected Lie group, $K$ a compact subgroup
of $G$, $\Gamma \subset G$ a torsion-free uniform lattice, and $\cH$ a subalgebra of the Lie algebra $\cG$ of $G$ contained in the normalizer
of the Lie algebra $\cK$ tangent to $K$ . Elements of $\cH$
can be seen as left invariant vector fields on $G$, hence induce an action of $H$ on $\Gamma\backslash G/K$, where $H$ is the simply connected Lie group with
Lie algebra isomorphic to $\cH$ (the immersion $H \looparrowright G$ is not assumed to be an embedding).
The actions defined in this way are called \textit{algebraic}. More precisely, our terminology will be to call $(G, K, \Gamma, \cH)$ an \textit{algebraic action}.
The necessary and sufficient condition for this action to be Anosov is easy
to infer (Proposition~\ref{pro:alganosov}):  for some element $\fh_0$ of $\cH$, the action induced by ${\rm ad}(\fh_0)$ on $\cG/(\cH \oplus \cK)$ must be hyperbolic,
ie. all the eigenvalues must have a non-vanishing real part.
In this paper, we will only consider the case where $H$ is nilpotent, with sometimes a special attention to the case $H=\RR^k$ with $k\geq1$.

There are essentially two families of algebraic Anosov actions of $\RR^k$:

-- \textit{Suspensions of an action of $\ZZ^k$ by automorphisms of nilmanifolds:} such a suspension is Anosov as soon as one of the automorphisms
is an Anosov diffeomorphism (\S~\ref{sub:exsus});

-- \textit{Weyl chamber actions:} when $G$ is a real semisimple Lie group of $\RR$-rank $k$, the centralizer of a (real split) Cartan subspace $\mathfrak a$
is a sum $\mathfrak a \oplus \cK$ where $\cK$ is tangent to a compact subgroup $K$. Then the right action of $\mathfrak a$ on $\Gamma\backslash G/K$
is Anosov (\S~\ref{sub:exweyl}). Actually, this phenomenon is notoriously at the very basis of the representation theory of semisimple Lie groups.

The two previous constructions can be combined: consider
a Weyl chamber Anosov action $(\Gamma\backslash G/K, \mathfrak a)$, and a nilmanifold $\Lambda\backslash N$. Any representation $\rho: G \to {\rm Aut}(N)$
such that $\rho(\Gamma)$ preserves $\Lambda$ defines a flat bundle over $\Gamma\backslash G/K$; the horizontal lift of the $\mathfrak a$-action
is still Anosov as soon as some $\rho(\gamma)$ induces an Anosov automorphism of $\Lambda\backslash N$. We get an algebraic Anosov action
on $(\Lambda \rtimes_\rho \Gamma)\backslash(N \rtimes_\rho G)/K$.
These examples can still be deformed, by modifying the flat bundle structure, ie. by replacing $\Lambda \rtimes_\rho \Gamma$
by any other lattice in $N \rtimes_\rho G$.  For a more detailed discussion, see \S~\ref{sub:exsusweyl}.

This procedure can be generalized in the case where $(G, K, \Gamma, \cH)$ is any algebraic Anosov action,
leading to the notion of \textit{nil-suspension over $(G, K, \Gamma, \cH)$} (Definition~\ref{def:nilsus}): $(\hat{G}, \hat{K}, \hat{\Gamma}, \hat{\cH})$ is a nil-suspension of
$(G, K, \Gamma, \cH)$ if there is an algebraic map $\hat{\Gamma}\backslash\hat{G}/\hat{K} \to \Gamma\backslash{G}/K$ which is a flat bundle admitting as fibers nilmanifolds such that $\cH$ is the projection of $\hat{\cH}$ - but
the fibers are non necessarily transverse to $\hat{\cH}$, which has a component tangent to the fibers. When this component tangent to the fiber is trivial, the nil-suspension is called \textit{hyperbolic} (Definition~\ref{def:hyperbolic}).

%It is expected that, when $k\geq2,$ every irreducible\footnote{We omit the definition of irreducible actions since it is irrelevant to
%the present paper, except for the present remark.} Anosov action of $\RR^k$ is conjugate to one of the actions listed above.

P. Tomter gave in \cite{tomter1, tomter2} the complete list of algebraic Anosov flows, ie. algebraic Anosov actions in the case where the acting nilpotent Lie group is simply $\RR$. The analysis in somewhat simplified 
by the obvious fact that a nil-suspension over an algebraic Anosov flow which is still a flow is necessarily hyperbolic. The present
paper includes a new proof in this case, somewhat simpler than in \cite{tomter1}:
\medskip

\noindent\textbf{Theorem \ref{thm:tomter}.}
\textit{Let $(G, K, \Gamma, \RR)$ be an algebraic Anosov flow. Then, $(G, K, \Gamma, \RR)$ is commensurable to either
the suspension of an Anosov automorphism of a nilmanifold, or to a hyperbolic nil-suspension over the geodesic flow
of a locally symmetric space of real rank one.}
\medskip

%Observe that Weyl chamber actions for real Lie groups of $\RR$-rank $1$ reduces to geodesic flows on locally symmetric spaces of real rank one.

In this paper, we extend this theorem to the case of any nilpotent Lie group $H$. 
It appears that in this case we get many more (interesting) examples, including the Weyl chamber actions.
Let us first consider the case when $G$ is solvable:

\medskip
\noindent\textbf{Theorem \ref{th:solvable}.}
\textit{Every algebraic Anosov action $(G, K, \Gamma, \cH)$ where $G$ is a solvable Lie group is
commensurable to a nil-suspension over the suspension of an Anosov $\ZZ^{p}$-action on the torus ($p \in \NN^*$).
In particular, up to commensurability, the compact Lie group $K$ is finite (hence trivial if connected).}
\medskip

We then consider the semisimple case.  One can easily produce new examples, starting from a Weyl chamber action $(\Gamma\backslash G/K, \mathfrak a)$:
 it may happen that $\mathfrak a \oplus \cK$ admits another splitting $\cH \oplus \cK'$ where $\cH$ is an abelian subalgebra commuting with the compact Lie
 algebra $\cK'$. Then $(\Gamma\backslash G/K', \cH)$ is still algebraic Anosov. A typical example is the lifting of a geodesic flow of a hyperbolic $3$-manifold
 to the bundle of unit vectors normal to the flow: what we get is an Anosov action of $\RR^2$ (better to say, of $\RR \times \SS^1$). We call these examples
 \textit{modified Weyl chamber actions}. 
 
 Observe that $(\Gamma\backslash G/K, \mathfrak a)$ is naturally a flat bundle over $(\Gamma\backslash G/K', \cH)$, with fibers tangent to the lifted action; but it is not a nil-suspension because $G/K'$ is not a quotient of $G/K$
 by a normal subgroup of $G$. For more details, see \S~\ref{sub:elliptic}.
We will prove:

\medskip
\noindent\textbf{Theorem \ref{th:ss}.}
\textit{Every algebraic Anosov action $(G, K, \Gamma, \cH)$ where $G$ is a semisimple Lie group is commensurable to
a modified Weyl chamber action.}
\medskip

Finally, we are left with the general situation:

\medskip
\noindent\textbf{Theorem \ref{th:mixed}.}
\textit{Let $(G, K, \Gamma, \cH)$ be an algebraic Anosov action, where $G$ is not solvable and not semisimple.
Then:}

\textit{-- either $(G, K, \Gamma, \cH)$ is commensurable to an algebraic Anosov action on a solvable Lie group (it happens when $L$ is compact),}

\textit{-- either $(G, K, \Gamma, \cH)$ is commensurable to a central extension over a (modified) Weyl chamber action (it happens when $G$ is reductive),}

\textit{-- or $(G, K, \Gamma, \cH)$ is commensurable to a nil-suspension over an algebraic Anosov action which is commensurable to
a reductive algebraic Anosov action, ie. a central extension over a (modified) Weyl chamber action.}
\medskip

\textit{Central extensions} are nil-suspensions which are totally not hyperbolic: the fibers of the associated
flat bundle over the suspended action are tori included in the orbits of the lifted action (see Definition~\ref{def:normal}).

The statement in the last case may appear sophisticated. But it is not: it cannot be replaced by the simpler statement
\textit{"$(G, K, \Gamma, \cH)$ is commensurable to a nil-suspension over
a reductive algebraic Anosov action"}.  The point is that until now we didn't give the precise definition of "commensurability".
If two algebraic actions are commensurable in the sense of Definition~\ref{def:equiv}, then they are \textit{dynamically commensurable}, in the sense that they
admit finite coverings which are smoothly conjugate. However, our notion of commensurability (Definition~\ref{def:equiv}) is finer; it means that one can go
from one algebraic action to the other by simple modifications on $G$, $K$, $\Gamma$ and $\cH$,  which are listed in \S~\ref{sub:commen}, and which obviously preserve the dynamically commensurable class, but where the conjugacy
is "algebraic".
In summary, one can replace the last item in the statement of Theorem \ref{th:mixed} by "$(G, K, \Gamma, \cH)$ is \textit{dynamically} commensurable to a nil-suspension over
a reductive algebraic Anosov action", but we would get a slightly less precise result.
See Remark~\ref{rk:pastop} and the example given in \S~\ref{sub:nilcentral} for a detailed discussion on that question.

This phenomenon is actually one of the facts that prevent us to consider here in full detail the rigidity question related to these algebraic Anosov actions. A striking property of (certain) Anosov actions is the $C^\infty$-rigidity property:

\medskip
\noindent{\textbf{Theorem} [Katok-Spatzier, \cite{KatSpatz1}]
\textit{Let $(G, K, \Gamma, \RR^k)$ be a \textbf{standard} algebraic action by the abelian group $\RR^k$ with $k\geq2$. Then, smooth actions of $\RR^k$ on $\Gamma\backslash{G}/K$ which are $C^1$-close to
the standard algebraic action are smoothly conjugate to it.}
\medskip

Standard algebraic actions of $\RR^k$ are suspensions of Anosov actions of $\ZZ^k$ on nilmanifolds which satisfy a certain irreducibility property, Weyl chamber actions of semisimple Lie groups of real rank $\geq2$
such that the lattice $\Gamma$ is irreducible, and "twisted symetric space examples" (in the terminology of \cite{KatSpatz1}), i.e. the examples $(N \rtimes_\rho G, K, \Lambda \rtimes_\rho \Gamma, \cH)$ introduced
above where the representation $\rho: G \to {\rm Aut}(N)$ satisfies some irreducibility condition. Let us mention that the general expectation in the field is that
Anosov actions of $\RR^k$ satisfying a suitable  irreducibility property are standard, i.e. algebraic, and there are several recent results in this direction (\cite{kalininspatz, fishkalininspatz}).

It would be interesting to do a systematic study, determining among algebraic Anosov actions which one are $C^\infty$-rigid, including the case where $\cH$ is nilpotent but not necessarily abelian. However, we have no
new ideas to propose here on this matter, outside some obvious remarks
(for example, central extensions are certainly not $C^\infty$-rigid), and considerations which would be
merely a rewriting of the results in \cite{KatSpatz1, KatSpatz2}.
%Furthermore, the study of general
%algebraic Anosov actions, which are in general nil-suspensions, include the analysis of representations
%of a semisimple Lie group (the underlying (modified) Weyl chamber action) into the automorphism group
%of a nilpotent Lie group. It is not a very elementary task.

Our original concern was about Anosov actions of $\RR^k$, but we soon realized that the case of actions
of nilpotent Lie groups was similar, even more natural. Indeed, a traditional ingredient in the study
of general Lie algebras are Cartan subalgebras (CSA in short), ie. nilpotent subalgebras equal to their own
normalizers. It is a classical result that every Lie algebra admits a CSA, and CSA's are precisely Engel subalgebras, ie. subalgebras $\mathfrak H \subset \cG$ such that the adjoint action of $\mathfrak H$ on $\cG/{\mathfrak H}$ has no $0$-eigenvalue (cf. Lemma~\ref{le:engel}). Now, if $(G, K, \Gamma, \cH)$ is an algebraic Anosov action, $\cH$ is closely related to
CSA's in $\cG$: for example, when $\cG$ is solvable, then $\cH$ \textit{is} a CSA in $\cG$, and hence is unique up
to conjugacy in $\cG$.

When $\cG$ is not solvable, the relation between $\cH$ and $\mathfrak H$ is not so direct: in that situation,
there is a one-to-one correspondence between CSA's in $\cG$ and CSA's in the Levi factor $\cL$ of $\cG$;
there is a well defined notion of \textit{hyperbolic CSA} in $\cG$, which is unique up to conjugacy to $\cG$.
The point is that for algebraic Anosov actions which are \textit{simplified}, ie. such that the compact group
$K$ is semisimple, the sum $\cH \oplus \fa_0$, where $\fa_0$ is a maximal torus in $K$, is a hyperbolic CSA,
hence unique up to conjugacy in $\cG$. For a more detailed discussion, see \S~\ref{sub:csach}.

 \subsection*{Notations}
We denote Lie groups by latin letters $G$, $H$, ... and Lie algebras by gothic letters $\cG$, $\cH$, ... The adjoint actions
are denoted by ${\rm Ad}: G \to \cG$ and ${\rm ad}: \cG \to \cG$.
Given a morphism $f: G \to H$, we denote by $f_\ast: \cG \to \cH$ the associated Lie algebra morphism.
The torus of dimension $\ell$ is denoted by $\TT^\ell$.

\subsection*{Organization of the paper}
In the preliminary section \ref{sec.preli} we give the precise definitions of (algebraic) Anosov
actions and of commensurability. In section \ref{sec:example} we give a detailed description of
the fundamental examples (suspensions, modified Weyl chamber actions). We also define nil-suspensions and nil-extensions. Nil-suspensions which are nil-extensions are precisely central extensions. 
This section include a study of the fact
that any nil-suspension is a sequence of central extensions and hyperbolic nil-suspensions which cannot be permuted.
Section \ref{sec:classification} contains the proofs of the classification Theorems \ref{th:solvable}, \ref{th:ss} and \ref{th:mixed}. It starts by a preliminary subsection \ref{sub:prelim} containing general results, in particular, a link between the acting algebra $\cH$ and CSA's in $\cG$ (Proposition~\ref{pro:CSA}).
The following subsections \ref{sub:solvablecase}, \ref{sub:sscase}, \ref{sub:mixedcase} are
devoted to the proof of the classification case, in respectively the solvable case, the semisimple case, and
the general case.
In section~\ref{sec:conclusion} we show how to briefly recover Tomter's results through our own study;
we check that codimension one algebraic Anosov actions satisfies the \textit{generalized Verjovsky conjecture,}
namely that they are (dynamically) commensurable to nil-extensions of either the suspension of an Anosov action
of $\ZZ^k$ on the torus, or the geodesic flow of a hyperbolic surface. We then briefly develop the relation
between Anosov actions and Cartan subalgebras.
The last section \ref{sec:appendice} is an appendix, where we have collected some classical algebraic
facts necessary for our study.

%%%%%%%%%%%%%
%%%%%%%%%%%%%
\section{Preliminaries}
\label{sec.preli}
%%%%%%%%%%%%%
%%%%%%%%%%%%%

%%%%%%%%%%%%%%%%%%%%%%%%%%%%%%%%%%%%%%%%%%
 \subsection{Algebraic Anosov actions}
 \label{sub.def}
%%%%%%%%%%%%%%%%%%%%%%%%%%%%%%%%%%%%%%%%%%%

Let $H$ be a simply connected nilpotent Lie group of dimension $k$, let $\cH$ be the Lie algebra of $H$, and let $M$ be a $C^\infty$ manifold of dimension $n+k$, endowed with a Riemannian metric $\|\cdot \|,$ and let $\phi$ be a locally free smooth action of
the simply connected Lie group $H$ on $M$. Let $T\cO$ the $k$-dimensional subbundle of $TM$
 that is tangent to the orbits of $\phi$.

 \begin{definition}\hspace{-.5cm}.
 \label{defi:defcodone}
 \begin{enumerate}
   \item We say that $a\in H$  is an \textbf{Anosov element} for $\phi$ if $g=\phi^a$
   acts normally hyperbolically with respect to the orbit foliation. That is, there exists real
  numbers $\lambda > 0,\ C > 0$ and a continuous $Dg$-invariant splitting of the tangent bundle
  $$
  TM=E_a^{ss}\oplus T\cO\oplus E_a^{uu}
  $$
  such that
  $$
  \begin{array}{cc}
   \|Dg^n|_{E_a^{ss}}\|\leq Ce^{-\lambda n} & \forall n>0 \\
   \|Dg^n|_{E_a^{uu}}\|\leq Ce^{\lambda n} & \forall n<0
  \end{array}
  $$
   \item Call $\phi$ an \textbf{Anosov action} if some $a\in H$ is
   an Anosov element for $\phi$.
  \end{enumerate}

 The action $\phi$ is a {\it codimension-one} Anosov action
 if $E^{uu}_a$ is one-dimensional for some Anosov element $a \in H$.
 \end{definition}

 Here, we are concerned with \textit{algebraic} Anosov actions of $H$.

\begin{definition}
\label{def:algebraicaction}
A (nilpotent) algebraic action is a quadruple $(G, K, \Gamma, \cH)$ where:

--  $G$ is a connected Lie group,

-- $K$ is a compact subgroup of $G$,

-- $\cH$ is a nilpotent Lie subalgebra of the Lie algebra $\cG$ of $G$ contained in the normalizer of $\cK$,
the Lie subalgebra tangent to $K$, and such that $\cH \cap \cK = \{ 0 \}$,

-- $\Gamma$ is a uniform lattice in $G$ acting freely on $G/K$.

If the group $K$ is trivial, we simply denote the action by $(G, \Gamma, \cH)$.
\end{definition}

The justification of this terminology is that, given such a data, we have a locally free right action of the Lie group $H$ associated
to $\cH$ on the quotient manifold $\Gamma\backslash{G}/K$. The algebraic action is said \textit{Anosov} if this right action
is ... Anosov!

\begin{proposition}
\label{pro:alganosov}
An algebraic action $(G, K, \Gamma, \cH)$ is Anosov if and only if there is an element $\fh_0$ of $\cH$
and an ${\rm ad}(\fh_0)$-invariant splitting $\mathcal{G}=\mathcal{U}\oplus \mathcal{S}\oplus \mathcal{K}\oplus
 \mathcal{H}$ of the Lie algebra $\cG$ of $G$ such that the eigenvalues of ${\rm ad} \fh_0|_{\mathcal{U}}$ (resp.  ${\rm ad} \fh_0|_{\mathcal{S}}$) have positive (resp. negative) real part.
\end{proposition}

\begin{remark}
Several authors include in the definition of Anosov actions the property that the center of the group contains an Anosov element (\cite{push, tavarez}).
This restriction is crucial in the case of general Lie groups in order to ensure the invariance by the entire group of the stable-unstable decomposition of a given Anosov element. However, this restriction, in the case of algebraic Anosov actions of nilpotent Lie groups is unnecessary for that purpose, as shown by
the following Lemma~\ref{le:niltrig}. %Moreover, several examples would be excluded by such a restriction, see ....
\end{remark}

\begin{lemma}
\label{le:niltrig}
Let $N$ be nilpotent Lie subgroup of $\op{GL}(V)$, where $V$ is a real vector space of finite dimension. For every element $x$ of $N$ let $\op{Spec}(x)$ denote the set of (complex) eigenvalues
of $x$. For every element $\lambda$ of $\op{Spec}(x)$ let $V_\lambda(x)$ denote the generalized $\lambda$-eigenspace of $x$, ie. the maximal subspace of $V$ which is $x$-invariant and such that
the only eigenvalue of the restriction of $x$ to $V_\lambda(x)$ is $\lambda$. Then, the spectral decomposition $V = \oplus_{\lambda \in \op{Spec(x)}} V_\lambda$ is $N$-invariant.
\end{lemma}

\begin{proof}
This Lemma is classical. See for example \cite[Theorem 9.1, p. 42]{onishchik}.
\end{proof}

\noindent\textit{Proof of Proposition~\ref{pro:alganosov}.} Let $(G, K, \Gamma, \cH)$ be an algebraic action. Select an ${\rm Ad}(K)$-invariant metric on $\cG$, and equip $G$ with the associated left invariant metric. It induces a
metric on $\Gamma\backslash G/K$. The tangent space at $K$ of $G/K$ is naturally identified with the quotient $\cG/\cK$. The adjoint action of every element $\fh$ of $\cH$ induces an action
on $\cG/\cK$, that we still denote by ${\rm ad}(\fh)$.

Assume that $(G, K, \Gamma, \cH)$ is Anosov, ie. that there is an Anosov element $\fh_0$ of $\cH$.
The stable and unstable bundles $E_{\fh_0}^{ss}$ and $E_{\fh_0}^{uu}$ lifts as $\cH$-invariant sub-bundles. Since the right action of $\cH$ commutes with the left action of $G$ on $G/K$,
these subbundles, which are unique, have to be preserved by left translations. Hence, they define an ${\rm ad}(\cH)$-invariant splitting $\cG/\cK = \bar{\cS} \oplus \bar{\cU} \oplus \bar{\cH}$, where $\bar{\cU}$ is expanded by
${\rm ad}(\fh_0)$, $\bar{\cS}$ contracted by ${\rm ad}(\fh_0)$, and such that the restriction of ${\rm ad}(\fh_0)$ to $\bar{\cH}$ is nilpotent. Since the restriction of ${\rm ad}(\fh_0)$ to $\cK$ has only purely imaginary eigenvalues,
the spectral decomposition of ${\rm ad}(\fh_0)$ on $\cG$ must be of the form $\cG = \mathcal{U}\oplus \mathcal{S}\oplus \mathcal{K}\oplus \mathcal{H}$ as required.

Inversely, assume that $\cG$ admits an ${\rm ad}(\fh_0)$-invariant splitting $\mathcal{G}=\mathcal{U}\oplus \mathcal{S}\oplus \mathcal{K}\oplus \mathcal{H}$ as stated in the Proposition.
According to Lemma~\ref{le:niltrig}, the splitting $\mathcal{G}=\mathcal{U}\oplus \mathcal{S}\oplus \mathcal{K}\oplus  \mathcal{H}$ is ${\rm ad}(\cH)$-invariant. The Proposition then follows easily.\qed

In the sequel, we will need to consider the universal covering. In doing so, we may unwrap $K$ in the universal covering of $G$
in a non-compact subgroup. Hence we need to introduce the following notion:

\begin{definition}
\label{def:algebraicaction2}
A generalized (nilpotent) algebraic action is a quadruple $(G, K, \Gamma, \cH)$ satisfying all the items of Definition~\ref{def:algebraicaction},
except the second item replaced by:

-- the intersection $K \cap \Gamma$ is a uniform lattice in $K$, contained in the center of $G$.
\end{definition}

\subsection{Commensurability}
\label{sub:commen}
Our first concern is the classification of Anosov actions of nilpotent Lie group $\cH$ up to finite coverings and topological conjugacy.
%We will therefore identify algebraic actions leading to the same action of $\cH$, up to finite coverings.
There are several obvious ways for two algebraic actions to be conjugated one to the other, or, more generally,
to be finitely covered by conjugated algebraic actions:

\begin{enumerate}
\item If $\Gamma'' = \Gamma \cap \Gamma'$ has finite index in $\Gamma$ and $\Gamma'$, then $(G, K, \Gamma'', \cH)$ is a finite covering of
$(G, K, \Gamma, \cH)$ and $(G, K, \Gamma', \cH)$.

\item If $K'$ is a finite index subgroup of $K$, then $(G, K', \Gamma, \cH)$ is a finite covering of $(G, K, \Gamma, \cH)$.

\item Let $G'$ be a Lie subgroup of $G$ containing $\Gamma$, such that $\cH$ is contained in $\cG'$ and such that $G$ is generated by the union $K \cup G'$: then the natural map $G'/(K \cap G') \to G/K$ is bijective, and $(G, K, \Gamma, \cH)$ is conjugate to $(G', K\cap G', \Gamma, \cH)$.

\item Let $\cH'$ be a Lie subalgebra of $\cH \oplus \cK$ supplementary to $\cK$, ie. the graph of a Lie morphism $\zeta: \cH \to \cK$.
Then $(G, K, \Gamma, \cH')$ is conjugate to $(G, K, \Gamma, \cH)$.

\item Assume that $K$ contains a subgroup $L$ which is normal in $G$. Consider the quotient map $p: G \to G/L$.
The image of $\cH$ under $p_\ast$ is a Lie subalgebra $\bar{\cH}$ in direct sum with $p_\ast(\cK)$. This map
induces a diffeomorphism $\bar{p}: \Gamma\backslash G/K \to p(\Gamma)\backslash(G/L)/(K/L)$ which
is a conjugacy between $(G/L, K/L, p(\Gamma), \bar{\cH})$ and $(G, K, \Gamma, \cH)$.

%\item Let $G'$ be a normal subgroup of $G$ containing $H$ and let $C$ be a (compact) subgroup of $K$ so that $G$ is the semidirect product $G' \rtimes C$. Let $K' = K \cap G$.
%Then the inclusion $G' \subset G$ induces

\item Finally, let $p: G \to G'$ be an epimorphism with discrete kernel contained in $\Gamma$. Then $p(K)$ is still compact, $p(\Gamma)$
is discrete in $G'$, and $p(\Gamma) \cap p(K)$ is trivial. The morphism $p$ induces a finite covering of $(G', p(K), p(\Gamma), p_\ast(\cH))$ by $(G, K, \Gamma, \cH)$.
\end {enumerate}

\begin{definition}
\label{def:equiv}
Two algebraic actions are \textbf{commensurable} if there is a sequence
of transformations using items $(1) - (6)$ going from one of them to the other.

If moreover the resulting actions are conjugate, then the actions are \textbf{equivalent}.
\end{definition}

\begin{remark}
If two algebraic actions are commensurable, then they admit finite coverings which are conjugate one to the other.
In particular, if one of them is Anosov, the other is Anosov too.
\end{remark}

\begin{remark}
\label{rk:connected}
Let $K_0$ denote the identity component of $K$. By item $(2)$, $(G, K, \Gamma, \cH)$ and $(G, K_0, \Gamma, \cH)$ are
commensurable. Hence, up commensurability, we can always assume, and we do, that $K$ is connected.
\end{remark}

\begin{remark}
\label{rk:noidealK}
Let $\cL$ be an ideal of $\cG$ contained in $\cK$ and maximal for this property. Then, by item $(5)$,
$(G/L, K/L, \bar{\Gamma}, \bar{\cH})$ and $(G, K, \Gamma, \cH)$ are commensurable, and there is no non-trivial
ideal of the Lie algebra of $G/L$ contained in $K/L$. Hence, up commensurability, we can always assume, and we do, that $\cK$ contains
no ideal of $\cG$.
\end{remark}

\begin{lemma}
\label{le:commute}
Up to equivalence we can assume that $K$ centralizes of $H$.
\end{lemma}

\begin{proof} \noindent
Let $\cK = \cK^\ast \oplus \cH_1$ be the splitting in semisimple (the Levi factor) and abelian (the radical) subalgebras.
This splitting is unique, and $\cH_1$ lies in the center of $\cK$. The adjoint action
of $\cH$ preserves this splitting.

-- The Lie algebra $\cH_1$ is tangent to a compact abelian subgroup $H_1$
of $G$, a torus, which is $H$-invariant. But the automorphism group of a torus is discrete, isomorphic to the group ${\rm GL}(\ell, \ZZ)$
for some integer $\ell$. Hence the action of $H$ is trivial, and $\cH$ lies in the centralizer of $\cH_1$.

-- Now consider the induced adjoint action ${\rm ad}: \cH \to {\rm Aut}(\cK^\ast)$. Since $\cK^\ast$ is semisimple,
every derivation is an inner derivation, ie. ${\rm Aut}(\cK^\ast) \approx \cK^\ast$.
Therefore, there is a Lie algebra morphism $\nu: \cH \to \cK^\ast$  such that ${\rm ad}(h)$ and
${\rm ad}(\nu(h))$ have the same adjoint action on $\cK^\ast$ for every $h$ in $\cH$.
Replace $\cH$ by the graph $\cH_\nu \subset \cH \oplus \cK^\ast$
of $-\nu$: we obtain an equivalent algebraic action $(G, K, \Gamma, \cH_\nu)$ for which $\cH_\nu$ lies in the centralizer of $\cK$.
\end{proof}

\section{Examples}
\label{sec:example}
%%%%%%%%%%%%%%%
%%%%%%%%%%%%%%

\subsection{The "irreducible" models}
In this section, we describe the main examples, from which all other algebraic Anosov actions can be constructed through nil-suspensions, which will be described
in the next section.

\subsubsection{Suspensions of  Anosov actions of $\mathbb{Z}^k$ on nilmanifolds}
\label{sub:exsus}
Let $N$ be a simply connected  nilpotent Lie group and  $\Lambda$ a torsion-free uniform lattice in $N$.
Let $\rho:\mathbb{R}^k\to {\rm Aut(N)}$ be a representation such that $\rho(a)$ is hyperbolic for some
$a\in \mathbb{R}^k$ and  $\rho(v)$ preserves $\Lambda$ for all $v$ in $\mathbb{Z}^k \subset \RR^k.$
Then $\rho$ induces an Anosov action of $\mathbb{Z}^k$ on the nilmanifold $\Lambda\backslash{N}$,
and the suspension is an Anosov action of $\mathbb{R}^k$ on a manifold which is a bundle over the
torus $\mathbb{T}^k$ with fiber a nilmanifold.

This example is algebraic: the group $G$ is the semi-direct product of $N$ by $\RR^k$ defined
by the morphism ${\rho}: \RR^k \rightarrow {\rm Aut(N)}$; the compact group $K$ is trivial, and
the lattice $\Gamma$ is a semidirect product of $\Lambda$ by $\ZZ^k$.

\begin{remark}
In these examples, there is an element $v$ of $\ZZ^k$ such that $\rho(v)$ induces an Anosov diffeomorphism on $\Lambda\backslash{N}$.
Observe that it is known that any Anosov diffeomorphism on a nilmanifold $\Lambda\backslash{N}$ is topologically
conjugate to a linear one, ie. induced by an automorphism of $N$ preserving the lattice $\Lambda$ (\cite{manning}).

It is actually conjectured that any Anosov diffeomorphism is topologically conjugate to a hyperbolic \textit{infra}-nilmanifold
automorphism (cf. for example \cite{franks}, page 63). These automorphisms are defined as follows: let $F$ be a finite group
of automorphisms of the nilpotent Lie group $N$, and let $\Delta$ be a lattice in the semi-direct product $G = N \rtimes F$.
Then $\Delta$ acts freely and properly on $N$, defining a \textit{infra-nilmanifold} $\Delta\backslash{N}$.
Let $f: G \to G$ be any automorphism preserving $\Delta$ and for which $f(N)=N$: it induces a diffeomorphism
on $\Delta\backslash{N}$  called a hyperbolic \textit{infra}-nilmanifold.
When $f_*: \mathcal{N} \to \mathcal{N}$ is hyperbolic, this diffeomorphism is Anosov.

Suspensions of algebraic actions of $\ZZ^k$ on an infra-nilmanifold by infra-nilmanifold automorphisms, including
an Anosov one, are examples of Anosov actions of $\RR^k$, but they are all commensurable to suspensions of
actions of $\ZZ^k$ by automorphisms on nilmanifolds.
\end{remark}

\subsubsection{Weyl chambers actions}
\label{sub:exweyl}
Let $G$ be a non-compact semisimple connected real Lie group \textit{with finite center,} with Lie algebra
$\cG$. Let $\Gamma$ be a torsion-free uniform lattice in $G$, and $\mathfrak a$ a
($\RR$-split) Cartan subspace of $G$ (see the Appendix, \S~\ref{sec:appendice}).
%, ie. a connected abelian subgroup of $G$ whose adjoint action on $\cG$ is diagonalizable over $\RR$, and which is maximal
%for this property.
The centralizer of $\mathfrak a$ in $\cG$ is a
sum $\mathfrak a \oplus \cK$, where $\cK$ is the Lie algebra of a compact subgroup $K \subset G$. The right action of $\mathfrak a$
induces a $\RR^k$-action on the compact quotient $M =
\Gamma\backslash{G}/K$. The adjoint action of $\mathfrak a$ on $\cG$ preserves a splitting~:

$$
\cG = \mathcal{K}  \oplus \mathfrak a \oplus \sum_{\alpha \in \Sigma} \mathcal{G}^\alpha
$$

where every $\alpha$ (the {\em roots}) are linear forms
describing the restriction of ${\rm ad}(a)$ on $\cG^\alpha$.
The regular elements $a$ of $\mathfrak a$ for
which $\alpha(a) \neq 0$ corresponding to the Weyl chambers of the Lie group,
form a union of open convex cones of Anosov elements according to our definition.

This family of examples was mentioned in \cite{Imhof}. They are called \textit{Weyl chamber flows} in \cite{KatSpatz1}, but
we prefer to call them \textit{Weyl chamber actions,} since \textit{flow} should be reserved to the case $k=1$.

\subsection{Nil-extensions}
\label{sub:tellex}

\begin{definition}
\label{def:extension}
An action $(\hat{M}, \hat{\phi})$ of a simply connected nilpotent Lie group $\hat{H}$ is a \textbf{nil-extension} of an action $(M,\phi)$ of $H$
if there is an epimorphism $r: \hat{H} \to H$ and a bundle $p: \hat{M} \to M$ such that:

-- the $\phi$-orbits  of the kernel $H_0$ of $r$ are the fibers of $p$,

-- the kernel $H_0$ contains a lattice $\Lambda_0$ which is the isotropy group of the $H_0$-action at every point $\hat{x}$ of $\hat{M}$,

-- $p: \hat{M} \to M$  is a $r$-equivariant principal $\Lambda_0\backslash H_0$ bundle: for every $a$ in $\hat{H}$ we have:
$$p \circ \hat{\phi}^a = \phi^{r(a)} \circ p$$

\end{definition}

In particular, $p$ is a principal $\Lambda_0\backslash H_0$-bundle. The fibers of $p$ are contained in the orbits of $\hat{\phi}$. Observe that a nil-extension is never faithful since every isotropy group contains the lattice $\Lambda_0$.

Clearly, a nil-extension of $(M, \phi)$ is Anosov if and only if $(M, \phi)$ is Anosov. If $a$ is an Anosov element of $H$ such that $\phi^a$
fixes a point $x$ in $M$, any lift $\hat{a}$ of $a$ in $\hat{H}$ induces on the nilmanifold $p^{-1}(x)$ a diffeomorphism which is partially
hyperbolic, with central direction tangent to the restriction of $\cH$ to the fiber.

Nil-extensions do not introduce new essential dynamical feature
of the initial algebraic action they are constructed from. Hence we consider them as "trivial" deformations.

When the kernel $H_0$ is free abelian, the nil-extension is called a \textbf{abelian extension} (compare with \cite{brintop, bringroup}). One can define
nil-extensions as iterated abelian extensions (cf. \S~\ref{sub:niltor}). Be careful: the converse is not true. Indeed, the iteration of abelian extensions
may lead to a non-nilpotent solvable group.

In the next subsection, given an algebraic Anosov action  $(G, K, \Gamma, \cH)$ we see two different ways to produce abelian extensions which are still algebraic.

\subsubsection{Modified Weyl chamber actions}
\label{sub:elliptic}

Let $(G, K, \Gamma,\mathfrak a)$ be a Weyl chamber action. In some cases, the compact Lie group $K$ is not semisimple:
its Lie algebra $\cK$ splits as a sum of an abelian ideal $\cT_e$ (the radical, which is also the nilradical) and a semisimple (compact) ideal $\cK^\ast$.
Actually, $\mathfrak a \oplus \cT_e$ is the (nil)radical of $\mathfrak a \oplus \cK$, and $\cK^\ast$ a Levi factor. Observe that since Levi factors are unique up to conjugacy,
and since $\mathfrak a \oplus \cT_e$ commutes with $\cK^\ast$, it is the unique Levi factor.

Let $K'$ be a connected closed subgroup of $K$ containing the Levi factor $K^\ast$ of $K$: it is tangent to $\cK^\ast \oplus \cK_e$,
where $\cK_e$ is a subalgebra of $\cT_e$. Let $\cH_e$ be any subalgebra in $\cT_e$ supplementary to $\cK_e$. Let $\cH'$ be the subalgebra $\mathfrak a \oplus \cH_e$. Then $(G, K', \Gamma, \cH')$ is Anosov, since the action of any element
of $\cH' \subset \mathfrak a \oplus \cT_e$ admitting as $\mathfrak a$-component a regular element is still hyperbolic transversely to $\cH' \oplus (\cK_e \oplus \cK^\ast) = \mathfrak a \oplus \cK$.
For example, one can select $\cK_e = \{ 0 \}$: then $K' = K^\ast$ and $\cH' = \mathfrak a \oplus \cT_e$: this choice lead to the \textit{simplification of the Weyl chamber action}
(see Definition~\ref{defi:simplification}).

The projection $p: \Gamma\backslash{G}/K' \to \Gamma\backslash{G}/K$ is a principal bundle, admitting as fibers the right orbits of the torus tangent to $\cK_e$.
It follows that the algebraic action  $(G, K', \Gamma, \cH')$ is an abelian extension of $(G, K, \Gamma, \mathfrak a)$. Observe that the simplification
$(G, K^\ast, \Gamma, \mathfrak a \oplus \cT_e)$ is an abelian extension over any other modified Weyl chamber action obtained by other choices of $\cK_e$, $\cH$.

Such a construction is possible only when $\cT_e$ is not trivial, ie. when $K$ is not already semisimple and admits indeed a non-trivial radical.
It is the case for example when $G={\rm SO}(p,p+2)$: it has $\RR$-rank $p$, and
the compact part $K$ of the centralizer of a $\RR$-split Cartan subgroup is isomorphic to ${\rm SO}(2)$. Therefore, the right action of $A \times K$ on $\Gamma\backslash{G}$
is an abelian extension of the Weyl chamber action on $\Gamma\backslash{G}/K$.

More generally, when $G$ admits $\ell$ simple factors of the form above, then the radical of $K$ contains a torus $\TT^\ell$, product of the subgroups
${\rm SO}(2)$ in each of the $\ell$ factors.

As a particular case, consider the real Lie group $G = {\rm SO}(1,3)$. A Cartan subspace of ${\mathfrak{so}}(1,3)$ is the set of matrices of the form:

$$\left(\begin{array}{cccc}
    0 & t & 0 & 0 \\
    t & 0 & 0 & 0 \\
    0 & 0 & 0 & 0 \\
     0 & 0 & 0 & 0 \end{array}\right)$$

The compact part $K$ of the centralizer is the Lie group isomorphic to ${\rm SO}(2)$ comprising the matrices:

$$\left(\begin{array}{cccc}
    1 & 0 & 0 & 0 \\
    0 & 1 & 0 & 0 \\
    0 & 0 & \cos\theta & \sin\theta \\
     0 & 0 & -\sin\theta & \cos\theta \end{array}\right)$$

It is well-known that the Weyl chamber flow $(G = {\rm SO}(1,3),K , \Gamma, \mathfrak a)$ is the geodesic flow on the unit tangent bundle $\Gamma\backslash {\rm SO}(1,3)/K$ over the hyperbolic manifold
$\Gamma\backslash {\rm SO}(1,3)/{\rm SO}(3) \approx \Gamma\backslash{\mathbb H}^3$. The only way to modify this Weyl chamber flow is to include the entire $\cK$ in the new $\cH'$,
hence to consider the algebraic action $(G, \Gamma, \cH' = \cH \oplus \cK)$ with no more compact subgroup.

It is an action of $\RR \times \SS^1$ on $\Gamma\backslash{\rm SO}(1,3)$. Observe that this homogeneous manifold is the unit frame bundle over the hyperbolic
manifold $\Gamma\backslash{\mathbb H}^3$, with is naturally identified with the bundle over $\Gamma\backslash{\mathbb H}^3$, made of pairs of orthogonal unit tangent
vectors $(u,v)$ (there is a unique $w$ in $\Gamma\backslash T^1{\mathbb H}^3$ such that $(u, v, w)$ is an oriented basis).
The action of the first component $\RR$ corresponds to the motion of $u$ along the geodesic flow, pushing $v$ along the geodesic by parallel transport, whereas the action of
the second component $\SS^1$
fixes the first vector $u$ and rotates the second vector $v$ in the $2$-plane orthogonal to $v$.

This example cannot be generalized for geodesic flows in higher dimensions, since the compact part of the centralizer of a Cartan subspace in ${\rm SO}(1,n+1)$ is
isomorphic to ${\rm SO}(n)$, which is not abelian for $n>2$.

\begin{remark}
\label{rk:nocompact}
Every modified Weyl chamber action is equivalent to a modified Weyl chamber action associated to semisimple Lie group with
no compact simple factors.

Indeed, let $(G, K, \cH, \Gamma)$ be a modified Weyl chamber action, where $G$ is simply connected.
Let $C$ be the product of the all the \textbf{compact} simple factors of $G$. Then $G = G_{nc}.C$ where $G_{nc}$
is a semisimple Lie group with no compact factor. Since $C$ is compact, the projections of $\cS$, $\cU$ in its Lie algebra $\cC$
are trivial; hence $\cC$ is the sum of the projections of $\cH$ and $\cK$. Since $C$ is semisimple, its radical is trivial, therefore
the projection of $\cH$ is trivial too: $\cC$ is actually the projection of the semisimple part $K^\ast$ of $K$. Hence,
$K$ contains the compact part $C,$ which is normal in $G$. Now by item (5) in the definition of commensurability, $(G, K, \cH, \Gamma)$ is equivalent
to $(G/C \approx G_{nc}, K/C, p_\ast(\cH), p(\Gamma))$ where $p$ is the projection $p: G \rightarrow G/C$.

\end{remark}

%We call these examples \textit{modified extensions of Weyl chamber actions} (because the new elements in the Lie group $H_1$ tangent to $\cH_1$ are elliptic, whereas
%every element of $\mathfrak a$ is hyperbolic).

\subsubsection{Central extensions}
\label{sub:extensioncentral}
In the previous section, abelian extensions were defined keeping the same Lie group $G$, but modifying the compact $K$. Here
we describe another procedure, keeping essentially the same compact group $K$, but enlarging the group $G$.
In the next section, this notion will be generalized to the notion of \textit{nil-suspensions.} 

\begin{definition}
\label{def:normal}
An algebraic action $(\hat{G}, \hat{K}, \hat{\Gamma}, \hat{\cH})$ is a \textbf{central extension} of $(G, K, \Gamma, \cH)$ if:
\begin{enumerate}
\item there is a central exact sequence
$$0 \rightarrow H_0 \rightarrow \hat{G} \rightarrow^p G \rightarrow 1$$
\item the Lie algebra $\cH_0$ of $H_0$ is contained in $\hat{\cH}$ (hence $\cH_0 \cap \hat{\cK} = \{ 0 \}$),
\item $p(\hat{K}) = K$,
\item $p(\hat{\Gamma}) = \Gamma$,
\item $p_\ast(\hat{\cH}) = \cH$.
\end{enumerate}
\end{definition}

\begin{lemma}
If $(\hat{G}, \hat{K}, \hat{\Gamma}, \hat{\cH})$ is a central extension of $(G, K, \Gamma, \cH)$
in the sense of Definition~\ref{def:normal}, then the associated action on $(\hat{\Gamma} \backslash \hat{G}/\hat{K})$ is
an abelian extension in the sense of Definition~\ref{def:extension} of the action on $\Gamma\backslash G/K$.
\end{lemma}

\begin{proof}
Since $p(\hat{K})=K$, we have an induced
map $\bar{p}: \hat{\Gamma} \backslash \hat{G}/\hat{K} \to \Gamma\backslash G/K$. This map is equivariant
with respect to the (right) actions of $\hat{\cH}$ and $\cH$.

Moreover, let $\hat{\Gamma} \hat{g} \hat{K}$ and $\hat{\Gamma}\hat{g}'\hat{K}$ be two elements
of $\hat{\Gamma} \backslash \hat{G}/\hat{K}$ having the same image under $p$. Then, there is an element
$\gamma$ of $\Gamma$ and an element $k$ of $K$ such that:
$$p(\hat{g})=\gamma p(\hat{g}')k$$
Now there is a $\hat{\gamma}$ in $\hat{\Gamma}$ and a $\hat{k}$ in $\hat{K}$ such that $p(\hat{\gamma})=\gamma$ and $p(\hat{k})=k$. Then:
$$p(\hat{g})=p(\hat{\gamma}\hat{g}'\hat{k})$$
Since $H_0$ is the kernel of $p$, we get that $\hat{\Gamma} \hat{g} \hat{K}$ and $\hat{\Gamma}\hat{g}'\hat{K}$ lies in the same
orbit of the right action of $H_0$. Conversely, the fibers of $\bar{p}$ contains the orbits
of this action of $H_0$.

Let $\Lambda$ be the intersection $\hat{\Gamma} \cap H_0$. Since $\hat{\Gamma}$ is a uniform lattice, and
since $p(\hat{\Gamma})=\Gamma$, $\Lambda$ is a lattice in ${\rm Ker}(p) = H_0$.
Let $F$ be the intersection $H_0 \cap \hat{K}$. Since $H_0$ is tangent to $\hat{\cH}$ and since $\hat{K}$ is compact, $F$ is a finite group.
The product $F\Lambda$ is a lattice $\Lambda_0$  in $H_0$.
Since $H_0$ is in the center of $\hat{G}$, the (right) action of $H_0$ induces an action of $H_0/\Lambda_0$ on $\hat{\Gamma} \backslash \hat{G}/\hat{K}$.

Let $h$ be an element of $H_0$ fixing some element $\hat{\Gamma} \hat{g} \hat{K}$.
Then, there is an element $\hat{\gamma}$ of $\hat{\Gamma}$ and some element $\hat{k}$ of $\hat{K}$ such that $\hat{g}h = \hat{\gamma}\hat{g}\hat{k}$.
Taking the image under $p$ we get that ${p}(\hat{\gamma})$ fixes some element of $G/K$ (namely, $p(\hat{g})K$). Since the action of $\Gamma$ on $G/K$ is free,
$\hat{\gamma}$ lies in $\Lambda$. Therefore, it commutes with $\hat{g}$, and we have $h=\hat{\gamma}\hat{k}$. In particular,
$\hat{k}$ lies in $F$.
We finally get $h \in \Lambda_0$, hence the action of $H_0/F\Lambda_0 \approx \TT^\ell$ is free.

In summary, $\bar{p}$ is an equivariant map whose fibers are the orbits of $\TT^\ell$, hence an equivariant principal $\TT^\ell$-bundle.
The lemma is proved.
\end{proof}

%\begin{remark}
%\label{rk:centralcentral}
%A central $\TT^k$-extension of a central $\TT^\ell$-extension is a central $\TT^{k+\ell}$-extension.
%\end{remark}

\begin{remark}
\label{rk:product}
A very particular situation, where all the hypothesis in Definition~\ref{def:normal} are trivially satisfied, is the \textbf{product} case, when
$\hat{G}$ is simply the product $H_0 \times G$: it amounts to simply multiply the algebraic Anosov action by a torus $\TT^\ell$.
\end{remark}

\begin{remark}
\label{rk:l+l}
A central extension of a central extension is not necessarily a central extension. Indeed,
the kernel $H_0$ of the the first central extension does not necessarily lift in the biggest group.
\end{remark}

\subsubsection{Non-product central extensions}
\label{sub:starkov}
Here we provide a typical example, inspired by an example due to Starkov (cf. \cite{KatSpatz1}), of a non-product central extension:

Let $\Hei$ be the Heisenberg group (of dimension 3): there is a non-split central extension:
$$0 \to \RR \to \Hei \to \RR^2 \to 0$$
A convenient way to describe $\Hei$ is to define it as the product $\RR \times \RR^2$ equipped with the operation:
$$(t, u).(t', u') = (t+t'+\omega(u,u'), u+u')$$
where $\omega: \RR^2\times\RR^2 \to \RR$ is a symplectic form. Here, we will consider a symplectic form $\omega$ taking integer values
on $\ZZ^2 \subset \RR^2$.

Let $\Hei_\ZZ$ denotes the lattice of $\Hei$ made of integer elements of $\RR \times \RR^2$.
Then for every $A$ in ${\rm SL}(2,\ZZ)$ the map $A^*: \Hei \to \Hei$ defined by $A^*(t,u) = (t, A(u))$
is a morphism preserving $\Hei_\ZZ$ (observe that $A^*$ is trivial on the center of $\Hei$).

If $A$ is hyperbolic, then the induced diffeomorphism $\bar{A}^*: \Hei_\ZZ\backslash\Hei \to \Hei_\ZZ\backslash\Hei$ is
partially hyperbolic, with central direction tangent to the orbits of the center.
Let $(M_0, \phi_0)$ be the suspension of $\bar{A}^*$. The flow $\phi_0$ commutes with the (periodic) flow induced by the center of $\Hei$. These flows
altogether define a $\RR^2$-action on $M_0$ which is Anosov. It is algebraic, with trivial compact group $K$ and as group $G$
an extension of $\Hei$ by $\RR$. It is moreover a central extension of the suspension of $\bar{A}: \ZZ^2\backslash\RR^2 \to \ZZ^2\backslash\RR^2$.
This central extension is not a product, since (the $3$-step solvable) group $\Hei \rtimes_{A^*} \RR$ is not isomorphic to the ($2$-step solvable) product $\RR \times (\RR^2 \rtimes_A \RR)$.

\subsection{Nil-suspensions \textit{per se}}
\label{sub:exsusnil}
Let $(G, K, \Gamma, \cH)$ be an algebraic Anosov action. Consider an exact sequence:
\begin{equation}
\label{eq:nilgroup}
1 \longrightarrow N \longrightarrow \hat{G} \longrightarrow G \longrightarrow 1
\end{equation}
where $N$ is a nilpotent Lie group. This exact sequence implies an exact sequence:
\begin{equation}
\label{eq:nilalg}
0 \longrightarrow \cN \longrightarrow \hat{\cG} \longrightarrow \cG \longrightarrow 0
\end{equation}
where $\cN$ is the (nilpotent) Lie algebra of $N$. Observe that we don't require (\ref{eq:nilgroup}) and (\ref{eq:nilalg}) to be split.

Assume that $K$ lifts as a compact subgroup $\hat{K}$ in $\hat{G}$ %(it is the case for example if $K$ is semisimple)
whose tangent Lie subalgebra $\hat{\cK}$ satisfies $\hat{\cK} \cap \hat{\cN} = \{ 0 \}$.

Let $\hat{\Gamma}$ be any uniform lattice in $\hat{G}$ whose projection in $\hat{G}/N \approx G$ is $\Gamma$.
The intersection $\Lambda := \hat{\Gamma} \cap N$ is then a lattice in $N$. Observe that the projection map $\hat{G} \to \hat{G}/N \approx G$ induces a bundle map $\hat{\Gamma}\backslash\hat{G}/\hat{K} \to \Gamma\backslash{G}/K$
whose fibers are the orbits of the right action of $N$ (observe that since $N$ is normal in $\hat{G}$, right orbits and left orbits of $N$
are the same). More precisely, the fibers are diffeomorphic to the nilmanifold $\Lambda\backslash{N}$,
where $\Lambda = \Gamma \cap N$.

Let $\hat{\cH}$ denote a nilpotent subalgebra of $\hat{\cG}$ such that the projection $p(\hat{\cH})$ is $\cH$.

\begin{definition}
\label{def:nilsus}
$(\hat{G}, \hat{K}, \hat{\Gamma}, \hat{\cH})$ is a nil-suspension over $(G, K, \Gamma, \cH)$.
\end{definition}

The bundle map $\hat{\Gamma}\backslash\hat{G}/\hat{K} \to \Gamma\backslash{G}/K$ is equivariant relatively to the actions of $\hat{\cH}$, $\cH$.
The orbits of the right action of $\hat{\cH}$ admit a part tangent to the fibers of this fibration: the orbits of $\hat{\cH} \cap \cN$. When this tangent
part is the entire fiber, ie. when $\hat{\cH}$ contains $\cN$, the nil-suspension is a nil-extension in the sense of Definition~\ref{def:extension}.
In particular, central extensions are precisely nil-suspensions for which the exact sequences (\ref{eq:nilgroup}), (\ref{eq:nilalg}) are central,
and such that $\cN$ is contained in $\hat{\cH}$.

An arbitrary nil-suspension of an algebraic Anosov action is not necessarily Anosov. The proof of the following proposition is obvious:

\begin{proposition}
\label{pro:nilanosov}
The algebraic action $(\hat{G}, \hat{K}, \hat{\Gamma}, \cH)$ is Anosov if and only if some element $\fh_0$ of $\cH$ admits a lift $\hat{\fh}_0$ in $\hat{\cH}$ such that the induced adjoint action on $\cN/(\hat{\cH} \cap \cN)$ is hyperbolic.\qed
\end{proposition}

\begin{remark}
Observe that for if $\fh_0$ admits a lift in in $\hat{\cH}$ satisfying the hypothesis of Proposition~\ref{pro:nilanosov}, then \textbf{all} the lifts of $\fh_0$ have the same property.
Indeed, the algebra $p^{-1}(\cH)$ is solvable, hence its adjoint action on $\cN \otimes \CC$ can be put in the triangular form. The diagonal coefficients provide linear maps $\lambda_1$, ... , $\lambda_n: p^{-1}(\cH) \rightarrow \CC$,
that vanish on $\cN$, and therefore induce maps $\bar{\lambda}_1$, ... , $\bar{\lambda}_n: \cH \rightarrow \CC$. Then the condition stated in Proposition~\ref{pro:nilanosov} is equivalent to the statement that
for every $i$, the complex number $\lambda_i(\fh_0)$ is either $0$ or have a non-vanishing real part.
\end{remark}

Some special cases deserve a particular terminology.

\begin{definition}
\label{def:hyperbolic}
Let $(\hat{G}, \hat{K}, \hat{\Gamma}, \hat{\cH})$ is an Anosov nil-suspension over $(G, K, \Gamma, \cH)$.
When the action on $\cN$ is hyperbolic, i.e. when $\cN \cap \hat{\cH}$ is trivial, the nil-suspension is \textbf{hyperbolic}.

When the nilpotent group $N$ is abelian, i.e. when the nilmanifold $\Lambda\backslash N$ is a torus, the nil-suspension
is a $\TT^\ell$-suspension.
\end{definition}

\subsubsection{Nil-suspensions as iterated $\TT^\ell$-suspensions}
\label{sub:niltor}
Every nil-suspension can be obtained by iteration of  $\TT^\ell$-suspensions.
Indeed, let $N_n$ be the last non-trivial term in the lower central serie of $N$, and let $\bar{N}$ be the quotient $N/N_n$.
Let $\bar{\rho}: \hat{G} \to {\rm Aut}(\bar{N})$ be the induced representation.
According to Theorem~\ref{th:latticenil} $\Lambda_n = \Lambda \cap N_n$ is a lattice in $N_n$.
It follows that the projection of $\hat{\Gamma}$ in $\hat{G}/N_n$
is a lattice $\hat{\Gamma}_n$. Then $(\hat{G}/N_n, \hat{K}/N_n, \hat{\Gamma}_n, \hat{\cH}/(\hat{\cH} \cap \cN_n))$ is
an Anosov nil-suspension over  $(G, K, \Gamma, \cH)$. Clearly, $(\hat{G}, \hat{K}, \hat{\Gamma}, \hat{\cH})$ is $\TT^\ell$-suspension
over $(\hat{G}/N_n, \hat{K}/N_n, \hat{\Gamma}_n, \hat{\cH}/(\hat{\cH} \cap \cN_n)$: the fiber is the torus $\Lambda_n\backslash N_n$,
and the representation is the one induced by the action of $\hat{G}$ on $N_n$ by conjugacy.

The claim follows by induction.

\begin{remark}
Proposition~\ref{pro:nilanosov} has a simpler formulation in the case of $\TT^\ell$-suspensions: in this case, the adjoint action of $\hat{\cG}$ on $\cN$
induces an action of $\cG$ on the abelian Lie algebra $\cN$, since inner automorphisms of $\cN$ are then trivial. Then, the condition is that
the action of $\cH$ on $\cN$ has to be partially hyperbolic, with central direction $\hat{\cH} \cap \cN$.

This criteria can be applied successively in the general case at each level of a tower of $\TT^\ell$-suspensions defining the general nil-suspension
\end{remark}

\begin{remark}
Here we give an example where it appears clearly that the operations of hyperbolic $\TT^\ell$-suspensions and central extensions
do not commute.

Let $V$ be a $3$-dimensional vector space, endowed with a basis $(e_0, e_1, e_2)$. Let $N$ be the product $V \times \bigwedge^2V$ equipped
with the composition law:
$$(u, \omega).(u', \omega') := (u+u', \omega + \omega' + u\wedge u')$$
$N$ is a nilpotent Lie group, with $[N, N] = \bigwedge^2V$.
Let $\Delta$ be the lattice in $V$ comprising $\ZZ$-linear combinations of $(e_0, e_1, e_2)$, and let
$\Delta^*$ be the lattice in $\bigwedge^2V$ comprising $\ZZ$-linear combinations of  the $e_i \wedge e_j$.
Then $\Delta \times \Delta^*$ is a lattice $\Gamma$ in the group $N$.

Let $A$ be any linear automorphism
of $V$ preserving $\Delta$. It induces a linear transformation on $\bigwedge^2V$ preserving $\Delta^*$.
We select $A$ so that it preserves $e_0$ and admits two eigenvector $u_+$, $u_-$,
associated respectively to real eigenvalues $\lambda$, $\lambda^{-1}$. Then the map ${\bf A}: (u, \omega) \mapsto (A(u), A^*(\omega))$ induces
a partially hyperbolic automorphism $f_A$ on the $2$-step nilmanifold $\Gamma\backslash{N}$. The central direction of ${\bf A}$ is the subgroup $H_0$ comprising
elements of the form $(a e_0, b u_+\wedge u_-)$ with $a$, $b \in \RR$.

Then the suspension of $f_A$, together with the right action of $H_0$, define an Anosov action of $\RR^3$. This suspension is a central extension over the Anosov action defined
in a similar way, but where $\bigwedge^2V$ is replaced by its quotient by the line spanned by $u_+ \wedge u_-$ (which indeed intersects $\Delta^*$ non-trivially).
This new Anosov action is itself a hyperbolic supension over the Anosov $\RR^2$-action defined as the combination of the suspension of the diffeomorphism
of  $\Delta\backslash V$ induced by $A$ and the translations in $\Delta\backslash V$ along the central direction $e_0$.

This last $\RR^2$-action is itself a central extension over the suspension of an Anosov automorphism of the $2$-torus.

However, the initial Anosov action of $\RR^3$ is not a central extension over the suspension of any Anosov action of $\ZZ^k$.
\end{remark}

\subsection{Nil-suspensions: the fundamental cases}
We have defined the notion of nil-suspension in the general case. Our main Theorem states that, up to commensurability,
every algebraic Anosov action is a nil-suspension over the suspension of an Anosov action of $\ZZ^k$ or over (a central extension of)
a modified Weyl chamber action. Even if dynamically poorly relevant, nil-suspensions are already by themselves quite complicated.
In this section, we try to provide more precisions about how can be nil-suspensions over these fundamental examples.

\subsubsection{Nil-suspensions over suspensions of actions of $\ZZ^k$}
\label{sub:nilZk}

Let us consider an algebraic action $(G, \Gamma, \cH)$, suspension of an Anosov action of $\ZZ^k$:
$G$ is a semidirect product $\RR^p \rtimes_\rho \RR^k$, where $\rho: \RR^k \to \op{GL}(p, \RR)$ is a morphism,
whose restriction to $\ZZ^k$ takes value in $\op{GL}(p, \ZZ)$, and such that $\rho(\RR^k)$ (and therefore $\rho(\ZZ^k)$)
contains a hyperbolic element. To produce such an action is already not so trivial, involving finite field extensions over $\QQ$
(see for example Example 3 in \cite{Bar-Maq}).

Now let $(\hat{G}, \hat{\Gamma}, \hat{\cH})$ a nil-suspension over $(G, \Gamma, \cH)$. There is an exact sequence:
$$1 \to N_0 \to \hat{G} \to G \to 1$$
where $N_0$ is a nilpotent (connected) Lie group. Let $N$ be the preimage in $\hat{G}$ of the normal subgroup $\RR^p$ of $G$.
We assume here that $N$ is the nilradical of $\hat{G}$; it actually follows from our classification Theorem in the solvable case (Theorem~\ref{th:solvable})
that we don't loose any generality in doing so.

Let $N_1=N$, $N_2 = [N_1, N_1]$, etc... denote the terms of the lower central serie of $N$. Let us consider the last non-trivial term of this serie.
Let $H_n$ be the intersection between $\hat{H}$ and $N_n$. Since $N_n$ is contained in the center of $N$, $H_n$
commutes with every element of $N$. Moreover, it is preserved by the adjoint action of $\hat{H}$, which induces an action
of $\RR^k \approx G/N$. This induced action is moreover unipotent (since $\hat{H}$ is nilpotent). Let
$H^0_n$ be the (non trivial) subspace in $H_n$ comprising common fixed points for this action of $\RR^k$. Every element of $H^0_n$
is in the center of $N$ and commutes with $\hat{H}$: it lies in the center of the entire $\hat{G}$.

The intersection $\Delta = \Gamma \cap \hat{H}$ is a lattice in $\hat{H}$. Let $h_0$ be an Anosov element in this intersection. The action
of $h_0$ by conjugacy induces an automorphism in $N_n$ which is partially hyperbolic, with central direction $H_n$, and fixing pointwise $H^0_n$.
According to Remark~\ref{rk:latticenil}, $H^0_n \cap \Lambda$ is a lattice in $H^0_n$. We can then take the quotient by $H^0_n/(H^0_n \cap \Lambda)$. It follows that the action is a central extension over an algebraic Anosov action such that $H^0_n$ has been deleted.

Repeating inductively this procedure, we obtain that $(\hat{G}, \hat{\Gamma}, \hat{\cH})$ is a nil-extension over  an algebraic Anosov action such that
$H_n$ is trivial (this nil-extension is not necessarily a central extension, cf. Remarque~\ref{rk:l+l}).

Now by taking the quotient by $N_n$ as in \S~\ref{sub:niltor} we get that the Anosov action is a $\TT^\ell$-suspension over another algebraic Anosov action,
for which the nilradical has now nilpotent index $n-1$. Observe that this suspension is
hyperbolic since $H_n = H \cap N_n$ is trivial.

By induction, we obtain that $(\hat{G}, \hat{\Gamma}, \hat{\cH})$ is obtained from $(G, \Gamma, \cH)$
by an alternating succession of nil-extensions and hyperbolic $\TT^\ell$-suspensions. Observe that this conclusion
does not hold for nil-suspensions over arbitrary algebraic Anosov actions: here, the fact that the quotient $\hat{G}/N$ is abelian
is crucial.

\subsubsection{Nil-suspensions over Weyl chamber actions}
\label{sub:exsusweyl}

A first remark is that when the algebraic action $(G, K, \Gamma, \cH)$ is a Weyl chamber action,
then the hypothesis of Proposition~\ref{pro:nilanosov} is easily satisfied.

More precisely: let $G$ be a semisimple Lie group, let $(G, K, \Gamma, \mathfrak a)$ be a Weyl chamber action, and let $\hat{G}$ be a nil-extension of
$G$:
\begin{equation}
\label{eq:eq}
1 \to N \to \hat{G} \to G \to 1
\end{equation}
Then $G$ is the semisimple part of $\hat{G}$, and $N$ is the nilradical of $\hat{G}$: the exact sequence \ref{eq:eq} is split.
In particular, up to finite covers, $\hat{G}$ is isomorphic to a semi-direct product $G \ltimes N$. Let $\rho: \cG \to {\rm GL}(\cN)$ be the representation induced by the adjoint representation.
Then, as a general fact for representations of semisimple real Lie groups,
for every $a$ in $\mathfrak A$, $\rho(a)$ is $\RR$-split (see \cite[Chapter 4, Proposition~4.3]{onishchik}, the point
is that any weight is a $\RR$-linear combination of roots).
Therefore, non-zero weights of $\rho$ are either positive or negative. In other words,
once fixed a Weyl chamber $C$ we have a splitting $\cN = \cN^{\mathfrak a} \oplus \cN^+ \oplus \cN^-$
such that $[\mathfrak a, \cN^{\mathfrak a}] = 0$, and such that for every element $a_0$ of $C$, the restriction of $a_0$ to $\cN^+$ (respectively $\cN^-$) has only positive (respectively negative)
eigenvalues. Define $\hat{\cH}$ as the sum $\cN^{\mathfrak a} \oplus \mathfrak a$ in $\hat{\cG} \approx \cN \oplus \cG$. Choose $\hat{K}$ as the copy of $K$ in the Levi factor $G$ of $\hat{G}$,
and let $\hat{\Gamma}$ be a cocompact lattice of $\hat{G}$ projecting on $\Gamma$:

\begin{proposition}
The nil-suspension $(\hat{G}, \hat{K}, \hat{\Gamma}, \hat{\cH})$ of $(G, K, \Gamma, \mathfrak a)$ is Anosov.\qed
\end{proposition}

\begin{remark}
\label{rk:katoksus}
$\TT^\ell$-suspensions over Weyl chamber actions was already introduced in \cite[Example 2.7]{KatSpatz1} where
they were called twisted symmetric space examples, but only in the {hyperbolic} case, i.e. when
the central direction $V^{\mathfrak a}$ is trivial. It was also observed that this construction can be iterated, alternating central extensions
and $\TT^\ell$-suspensions.
\end{remark}

Actually, the main difficulty in constructing nil-suspensions over Weyl chamber actions is to find a representation $\rho$ such that $G \ltimes_\rho N$ admits indeed
a cocompact lattice $\hat{\Gamma}$.
The typical way to do so is to consider a $\QQ$-algebraic structure one $G$, i.e. to consider
$G$ as the real form of $\QQ$-algebraic group $\bf G$, such that $G_\ZZ$ (which, by Borel-Harish-Chandra Theorem is always a lattice)
is cocompact, and a $\QQ$-representation $\rho$ from ${\bf G}$ into a $\QQ$-vector space ${\bf V}$.
Then, $\rho(G_\ZZ)$ preserves a lattice $\Delta$ of ${\bf V}_\QQ$.
Take then as nilpotent Lie group $N$ the real form $V := {\bf V}_\RR$ and as cocompact lattice $\hat{\Gamma}$
of $\hat{G} = G \ltimes_\rho V$ the subgroup $\Gamma \ltimes_\rho \Delta$.

\begin{definition}
\label{def:arithmeticsuspension}
$(G \ltimes_\rho V, K, \Gamma \ltimes_\rho \Delta, {\mathfrak a} \ltimes_\rho V^{\mathfrak a})$ is an \textit{arithmetic} $\TT^\ell$-suspension over $(G, K, \Gamma, {\mathfrak a})$.
\end{definition}

One can also deform the lattice $\hat{\Gamma} = \Gamma \ltimes_\rho \Delta$; after such a deformation we still have an algebraic Anosov action.

A special way to do so is the following: let $\tau: \Gamma \to V$ be any $1$-cocycle, i.e.
a map satisfying:
$$ \forall \gamma, \gamma' \in \Gamma, \;\; \tau(\gamma\gamma') = \tau(\gamma) + \rho(\gamma)(\tau(\gamma'))$$

Then:
$$\begin{array}{ccc}
\Gamma \ltimes_\rho \Delta & \rightarrow & G \ltimes_\rho V \\
(\gamma, \delta) & \mapsto & (\gamma, \delta + \tau(\gamma))
\end{array}$$
is a morphism, whose image is a cocompact lattice of $G \ltimes_\rho V$, that we denote by $\hat{\Gamma}_\tau$.

\begin{definition}
$(G \ltimes_\rho V, K, \hat{\Gamma}_\tau, {\mathfrak a} \ltimes_\rho V^{\mathfrak a})$ is a \textit{deformation of an arithmetic} $\TT^\ell$-suspension over $(G, K, \Gamma, {\mathfrak a})$.
\end{definition}

Observe that cohomologous cocycles produce conjugated lattices, hence equivalent algebraic Anosov actions.

Geometrically, the fact that $\tau$ vanishes in $H^1(\Gamma, V)$ (the non-deformed case) means that there is a Levi factor $\approx G$ whose
intersection with $\hat{\Gamma}_\tau$ is a lattice in itself. In other words, the flat bundle over $\Gamma\backslash G/K$ admits then
a horizontal section (the projection in $\Gamma\backslash\hat{G}/K$ of a right orbit in $G \ltimes_\rho V$
of a Levi factor), which, in particular, is preserved by the action. In this case, there is an equivariant copy of the Weyl chamber action in the suspended action.

\begin{proposition}
\label{pro:arithmetic}
Assume that $\Gamma$ is an irreducible lattice in $G$ and that the real rank of $G$ is at least $2$ and that $G$ has no compact simple factor.
Then, any $\TT^\ell$-suspension over $(G, K, \Gamma, \cA)$ is
commensurable to a deformation of an arithmetic $\TT^\ell$-suspension.

If moreover $G$ has no simple factor of real rank $1$, then any $\TT^\ell$-suspension over $(G, K, \Gamma, \cA)$ is
commensurable to an arithmetic $\TT^\ell$-suspension.
\end{proposition}

\begin{proof}
%Let $(G \ltimes_\rho V, K, \hat{\Gamma}_\tau, \cA \ltimes_\rho V^A)$ be a $\TT^\ell$-suspension over $(G, K, \Gamma, \cA)$.
%Let $\Delta$ be the lattice $\hat{\Gamma} \cap V$. Every element $\hat{\gamma}$ of $\hat{\Gamma}$ can be written as an
%element $(\gamma, v(\hat{\gamma}))$ of $\Gamma \ltimes_\rho V$, where $\gamma$ is an element of $\Gamma$.
%
By Margulis arithmeticity Theorem, under the hypothesis of the proposition, there exists a $\QQ$-algebraic group ${\bf H}$ and a morphism $r: H_\RR \to G$ with compact kernel
such that, up to finite index,  we have $\Gamma = r(H_\ZZ)$. Moreover, $\rho \circ r(H_\ZZ)$ preserves the lattice $\Delta=\hat{\Gamma} \cap V$ this lattice defines
a $\QQ$-structure on $V$ such that $\Delta = V_\ZZ$. Then $\bar{\rho} = \rho \circ r$ is a $\QQ$-morphism. Hence, $(H_\RR \ltimes_{\bar{\rho}} V, K, H_\ZZ \ltimes_{\bar{\rho}} \Delta, \cA \ltimes_{\bar{\rho}} V^A)$ is an arithmetic $\TT^\ell$-suspension.

Now, by a Theorem of Mostow (see \cite[chapter 4, Theorem 2.3]{vinberg2}), (a finite index of)
the lattice $\hat{\Gamma}$ can be deformed continuously into (a finite index lattice of) the lattice $H_\ZZ \ltimes_{\bar{\rho}} \Delta$.
The first part of the proposition follows.

For the second part of the proposition, see the discussion at the end of \cite[Chapter 4, \S~2]{vinberg2}.

\end{proof}

\begin{remark}
\label{sub:simplecompact}
In Remark~\ref{rk:nocompact}, we have observed that any Weyl chamber action is equivalent to a Weyl chamber action on a semisimple group
that has no compact simple factor. One is tempted to remove in Proposition~\ref{pro:arithmetic} the "no compact simple factors" assumption, since it seems
to be automatically satisfied.

However, it is not correct: let $G = G_{nc}.C$ be the decomposition of $G$ as a product of a semisimple Lie group without compact factor $G_{nc}$
and the product $C$ of all compact simple factors. Even if automatically included in $K$, the compact subgroup $C$ may act non-trivially on $V$,
and the components in $C$ of elements of $\Gamma$ plays a role in the definition of the suspension which cannot be removed by simply taking a finite index subgroup.
\end{remark}

\begin{remark}
A $\TT^\ell$-suspension over a Weyl chamber flow is not necessarily an abelian extension over a hyperbolic nil-suspension.
A typical example is provided when $\hat{G}$ is the group of \textit{affine transformations on $\cG$}, i.e.
the semidirect product $G \ltimes_{\rm Ad} \cG$ (indeed, this semidirect product can be seen as the subgroup of ${\rm Diff}(\cG)$ generated by
the transformations ${\rm Ad}(g)$ and translations on $\cG$). In this case, the invariant subspace $\cG^A$ is the subalgebra $\cA$, and it can be easily proved
that most orbits of $\cA$ in $(G_\ZZ \ltimes_{\rm Ad} \cG_\ZZ)\backslash(G \ltimes_{\rm Ad} \cG)/K$ are not compact.
\end{remark}

\subsubsection{Nil-suspensions over central extensions of Weyl chamber actions}
\label{sub:nilcentral}

As we will prove (Theorem~\ref{th:mixed}) any abelian extension over a modified Weyl chamber action is actually a central extension over (maybe another) modified Weyl chamber action, and
algebraic Anosov actions which are not as described in section \ref{sub:nilZk} are nil-suspensions over Anosov actions \textit{commensurable} to
a central extension over a (modified) Weyl chamber action.

Let $(\hat{G}, \hat{K}, \hat{\Gamma}, \hat{\cH})$ be such a central extension of a Weyl chamber action $(G, K, \Gamma, \mathfrak a)$
with associated central exact sequence $1 \to A \to \hat{G} \to G \to 1$.
Actually, this sequence is necessarily split, ie. $\hat{G} \approx A \times G$,
but this splitting is not necessarily preserved by $\hat{\Gamma}$.

The construction of a nil-suspension over $(\hat{G}, \hat{K}, \hat{\Gamma}, \hat{\cH})$
tantamount essentially to the construction of a nil-suspension
over $(G, K, \Gamma, \mathfrak a)$ (for example as described in the previous section \ref{sub:exsusweyl}) and to introduce a action
of $\ZZ^\ell$ (if $A$ has dimension $\ell$) commuting with the action on one fiber of the nil-suspension (which is a nilmanifold) induced
by the holonomy group on a compact orbit of $\fa$. We don't enter in any further detail; since we have no hope to give a complete description.

However, we want to stress out the following phenomenon: there exists central extensions of Weyl chamber actions, one commensurable one to the other,
but such that one admits (abelian hyperbolic) nil-suspension which is not commensurable to any nil-suspension over the other.
%$(G, K, \Gamma, \mathfrak a)$, and admitting nil-suspensions that are not commensurable to any nil-suspension over $(G, K, \Gamma, \mathfrak a)$.
It shows that some care is needed when replacing an algebraic action by another one which is commensurable to it.

Take as semisimple Lie group $G$ the group $\op{SU}(1,2)$. Let $\Gamma_0$ be a finite index subgroup of
a uniform arithmetic lattice of $\op{SU}(1,2)$; we can furthermore assume that $\Gamma_0/[\Gamma_0, \Gamma_0]$ is infinite, i.e. that there is a non-trivial
morphism $u: \Gamma_0 \to \RR$ (cf. \cite[page 98]{vinberg2}).
Since $\Gamma_0$ is an arithmetic lattice, there is a representation $\rho: G  \to \op{GL}(k, \RR)$
such that $\rho(\Gamma_0)$ preserves the lattice $\Lambda = \ZZ^k$ of $V = \RR^k$.
We select here $\rho$ so that it is hyperbolic, ie. such that it has no zero weight.
Let $V^\CC$ and $\rho^\CC: G \to \op{GL}(V^\CC)$ be the complexifications: $V^\CC$ contains a $\rho^\CC(\Gamma_0)$-invariant lattice $\Lambda^\CC$.

We consider the circle $\SS^1$ as the group
of complex numbers of norm $1$. Select a non trivial morphism $u: {\Gamma} \to \SS^1$ with dense image, and define the morphism $r_u: {\Gamma} \to \op{U}(1,2)$ by $r_u(\gamma) = u(\gamma)\gamma$, where the second factor
$u(\gamma)\gamma$ denotes the composition of $\gamma$ by the complex multiplication by $u(\gamma)$.
If $u$ is small enough, then the image $\Gamma_u$ of $r_u$ is a uniform lattice in $\op{U}(1,2)$.
Observe that $\rho^\CC(\Gamma_u)$ does not preserve $\Lambda^\CC$.
Actually, we can select $u$ so that there is no complex representation of $\op{U}(1,2)$ such that the image
of $\Gamma_u$ by this representation preserves a lattice: indeed, the space of such representations up to conjugacy is countable,
whereas there are uncountably many morphisms $u$.

Let $\mathfrak{u}(1,2)$ be the Lie algebra of $\op{U}(1,2)$, let
let $\mathfrak{su}(1,2)$ be the Lie algebra of $\op{SU}(1,2)$, considered as a subalgebra of $\mathfrak{u}(1,2)$: we can write
$\mathfrak{u}(1,2) = \cZ \oplus \mathfrak{su}(1,2)$, where $\cZ$ is the center of $\mathfrak{u}(1,2)$, tangent to the
subgroup $Z \approx \SS^1$ comprising matrices of multiplication by an element of $\SS^1$.

Select a Cartan subspace $\fa$ of $\mathfrak{u}(1,2)$, and consider it as a $1$-dimensional subalgebra of $\mathfrak{su}(1,2)$.
Let $K_0$ be the compact complement of the centralizer of $\fa$ in $\op{SU}(1,2)$,
that we can consider as a subgroup in $\op{U}(1,2)$, and let $\cH$ be the sum $\cZ \oplus \fa$.
Clearly, $(U(1,2), K_0, \Gamma_u, \cH)$ is a central extension of $(SU(1,2), K_0, \Gamma_0, \fa)$.

Consider now the group $G'= K' \times \op{U}(1,2)$, where $K'$ is the circle $\SS^1$, and let $\Gamma'$ be the image of the morphism
$\Gamma_0 \to G'$ mapping $\gamma$ on $(u(\gamma)^{-1}, r_u(\gamma))$. The center of $G'$ is the torus
$K' \times Z$. Clearly, $\Gamma'$ is a uniform lattice of $G'$. Moreover, $(G', K' \times K_0, {\Gamma}', \cH)$ is commensurable, even
equivalent, to $(\op{U}(1,2), K_0, \Gamma_u, \cH)$: the equivalence is obtained by dividing by the compact center $K'$ (item $(5)$).

Now the key point is that in doing so, we have introduced a new circle, which provides one more parameter
for the construction of a nil-suspension. More precisely:
Let $\hat{G}$ be the semidirect product $V^\CC \rtimes (K' \times \op{U}(1,2))$ where $\op{U}(1,2)$ acts on $V^\CC$ through
$\rho^\CC$, and $K' \approx \SS^1$ by multiplication by complex numbers of norm $1$. We have constructed $\Gamma'$ so that
its $K'$-component cancel the deformation we have introduced for $\Gamma_u$, so that now the action of $\Gamma' \subset G'$
preserves the lattice $\Lambda^\CC$ in $V^\CC$: the semidirect product $\hat{\Gamma}:= \Lambda^\CC \rtimes \Gamma'$ is a uniform lattice
in $\hat{G}$. We are in the situation very similar to the one considered in Definition~\ref{def:arithmeticsuspension}:
$(\hat{G}, K' \times K_0, \hat{\Gamma}, \cH)$  is then an abelian nil-suspension over $(G', K' \times K_0, {\Gamma}', \cH)$. Observe
that it is Anosov since $\rho$ has been selected hyperbolic.

On the other hand, since we have selected $\Gamma_u$ so that it preserves no lattice in a $\op{U}(1,2)$-complex vector space,
$(\op{U}(1,2), K_0, \Gamma_u, \cH)$ admits no abelian nil-suspension at all. Our claim follows.

\section{Classification of algebraic Anosov actions}
\label{sec:classification}
%%%%%%%%%%%%%%%%%%%%
%%%%%%%%%%%%%%%%%%%

 In all this section, $(G, K, \Gamma, \cH)$  denotes a fixed algebraic Anosov action, and $\fh_0$ an Anosov element of $\cH$. We assume
 that $K$ is connected (cf. Remark~\ref{rk:connected}). Up to conjugacy, one can assume that the $\cH$-orbit of $\Gamma$ in $G/K$ is compact, i.e. that $\Gamma \cap HK$ is a lattice in $HK$, that we denote by $\Delta$.

 We denote by $R$ the (solvable) radical of $G$, by $N$ its nilradical. We will need to consider the universal covering
 $\pi: \tilde{G} \to G$. The kernel of $\pi$ is a discrete group $\Pi$ contained in the center of $\tilde{G}$.
 The Lie algebra of $\tilde{G}$ is $\cG$.
 The group $\tilde{G}$ is a semidirect product $\tilde{R} \rtimes \tilde{L}$ of its radical $\tilde{R}$, which is tangent to $\cR$, and a Levi factor $\tilde{L} \approx \tilde{G}/\tilde{R}$.
 We have $\pi(\tilde{R}) = R$, and $\pi(\tilde{N}) =N$, where $\tilde{N}$ is the nilradical of $\tilde{G}$, which is tangent to the nilradical $\cN$
 of $\cG$.  We also denote by $\tilde{K}$ the unique subgroup of $\tilde{G}$ tangent to the subalgebra $\cH \subset \cG$. It is not
 necessarily a compact subgroup of $\tilde{G}$, but the intersection $\Pi' = \Pi \cap \tilde{K}$ is a lattice in $\tilde{K}$.
Hence, $(\tilde{G}, \tilde{K}, \tilde{\Gamma}, \cH)$ is a generalized algebraic Anosov action in the sense of Definition~\ref{def:algebraicaction2}.

 Since we also assume that $\Delta = \Gamma \cap HK$ is a lattice in $HK$, the intersection $\tilde{\Gamma} \cap \tilde{H}\tilde{K}$ is a lattice $\tilde{\Delta}$ in $\tilde{H}\tilde{K}$ too.

\subsection{Preliminary results}
\label{sub:prelim}
%%%%%%%%%%%%%%%%%%%%%%
\begin{lemma}
\label{le:commutator}
The stable and unstable subalgebras $\cS$ and $\cU$ lie in the derived ideal $[\cG, \cG]$.
\end{lemma}

\begin{proof}
The restrictions ${\rm ad} \fh_0|_{\mathcal{U}}$ and ${\rm ad} \fh_0|_{\mathcal{S}}$ have trivial kernel, hence are surjective.
Every element of $\cS$ and $\cU$ is therefore a commutator $[\fh_0, {\mathfrak h}]$. It proves the Lemma.
\end{proof}

A fundamental remark is that algebraic Anosov actions are closely related to the general notion of Cartan subalgebras (CSA) (see section \ref{sub:CSA}).
More precisely:

\begin{proposition}
\label{pro:CSA}
Let $\cK = \cK^\ast \oplus \cT_1$ be a Levi decomposition of $\cK$, and let $\fa_0$ be a maximal abelian subalgebra of $\cK^\ast$. Then,
$\cH \oplus \fa_0 \oplus \cT_1$ is a CSA of $\cG$.
\end{proposition}

\begin{proof}
Observe that the radical of the compact Lie group $K$ is abelian, hence $\cT_1$ is abelian.
More precisely, $\cT_1$ lies in the center of $\cK$, and the splitting $\cK = \cT_1 \oplus \cK^\ast$ is ${\rm ad}(\cK)$-invariant.
Since every element of $\cH$ commutes with every element of $\cK$, $\cH \oplus \fa_0 \oplus \cT_1$ is nilpotent. Let $\op x$ be an element in
the normalizer $\fN(\cH \oplus \fa_0 \oplus \cT_1)$. Since the decomposition $\cG = \cH \oplus \cK \oplus \cS \oplus \cU$ is ${\rm ad}\fh_0$-invariant,
it follows that $\op x$ must lie in $\cH \oplus \cK$: there are $\op y$, $\op z$, in $\cH$, $\cK$ respectively such that $\op x = \op y \oplus \op z$.
Since $\cH \subset \fN(\cH \oplus \fa_0 \oplus \cT_1)$, the $\cK$-component $\op z$ lies in $\fN(\cH \oplus \fa_0 \oplus \cT_1)$.
But since ${\rm ad} \op z$ vanishes on $\cH \oplus \cT_1$, it follows that the projection of $\op z$ in $\cK^\ast \approx \cK/\cT_1$ must normalize $\fa_0$. Since $\fa_0$ is a maximal abelian subalgebra
of $\cK^\ast$, hence a CSA of $\cK^\ast$ (\cite[Corollary p. 80]{humphreys}), the projection of $\op z$ in $\cK^\ast$ lies in $\fa_0$,
hence $\op z$ lies in $\cT_1 \oplus \fa_0$.

We have proved that $\cH \oplus \fa_0 \oplus \cT_1$ is its own normalizer. The proposition is proved.
\end{proof}

The construction described in \S~\ref{sub:elliptic} applies in general:
\begin{definition}
\label{defi:simplification}
The algebraic action $(G, K, \Gamma, \cH)$ is \textbf{simplified} if $K$ is semisimple.
A \textbf{simplification} of $(G, K, \Gamma, \cH)$ is an algebraic action $(G, K^\ast, \Gamma, \cH^\ast)$ where
$K^\ast$ is a Levi factor of $K$, and $\cH^\ast$ is the sum $\cH_\nu \oplus \cT_1$ where $\cT_1$ is the (nil)radical the Lie algebra of $K$
and $\cH_\nu$ the graph of a morphism $\nu: \cH \to \cK^\ast$ contained in the centralizer of $\cK$.
\end{definition}

Hence, $(G, K, \Gamma, \cH)$ admits an abelian extension by a simplified algebraic action.

\begin{lemma}
\label{le:simplyconnected}
Assume that $(G, K, \Gamma, \cH)$ is simplified. Then up to commensurability, one can assume that $G$ is simply-connected.
\end{lemma}

\begin{proof}
Since $K$ is semisimple, the Lie algebra $\cK$ is compact semisimple. The connected Lie subgroup $\tilde{K}$
of $\tilde{G}$ tangent to $\cK$ is therefore compact. Observe that $K$ is the image of $\tilde{K}$ by $p$. The Lemma follows from
item $(6)$ of the definition of commensurability.
\end{proof}

We will also need the following Lemmas, which hold in all cases, simplified or not:

\begin{lemma}
\label{le:KpasdansN}
Up to commensurability, $\cK \cap \cN$ is trivial.
\end{lemma}

\begin{proof}
Since $N$ is nilpotent, the compact ${N} \cap K$ is abelian. At the universal cover level, the identity component $\tilde{K}_0$ of the intersection $\tilde{N} \cap \tilde{K}$ is a free abelian subgroup whose intersection with $\Pi$ is a lattice $\Pi_0$ in $\tilde{K}_0$. According to \cite[Theorem 2.7, page 45]{vinberg2},
the inclusion $\tilde{K}_0 \hookrightarrow \tilde{N}$ is the unique Lie morphism extending the inclusion $\Pi_0 \hookrightarrow \tilde{N}$.

But $\Pi_0 \subset \Pi$ is contained in the center of $\tilde{G}$, hence, for every element $g$ of $\tilde{G}$, the inclusion
$g\tilde{K}_0g^{-1} \hookrightarrow \tilde{N}_0$ is also an extension of $\Pi_0 \hookrightarrow \tilde{N}$, hence must coincide
with $\tilde{K}_0 \hookrightarrow \tilde{N}$. Therefore, $\tilde{K}_0$ is normal in $\tilde{G}$. The lemma follows from Remark~\ref{rk:noidealK}.
\end{proof}

\begin{lemma}
\label{le:sanscentre}
Up to commensurability, $(G, K, \Gamma, \cH)$ is a central extension over an algebraic Anosov action $(G', K', \Gamma', \cH')$ such that the center of $G'$ is trivial.
\end{lemma}

\begin{proof}
The main argument is the one used in the proof of Lemma~\ref{le:centrefini} in the appendix.
Let $Z$ be the center of $G$. Consider any left invariant metric on $G$, and the induced metric on the compact quotient $\Gamma\backslash G$.
The (left or right) translation by $Z$ induces a group of isometries, all commuting with the right translations by elements of $G$. By Ascoli-Arzela
Theorem, the closure of this group of isometries is a compact abelian subgroup $T$ of $\op{Isom}(\Gamma\backslash G)$, whose elements still
commute with right translations by elements of $G$. Therefore, $T$ is the product of a torus $T_0$ by a finite abelian group, which is
induced by left translations by a closed subgroup $A$ of $G$, such that $A \cap \Gamma$ is a lattice in $A$. Let $A_0$ be the identity component of $A$.
We have a central exact sequence:
$$ 0 \to A_0 \to G \to G^\sharp \to 0$$
where $G^\sharp$ is the quotient $G/A_0$.
Let $\mathfrak A$ be the Lie algebra of $A_0$. Clearly, no element of $\mathfrak A$ may have a component in $\cS \oplus \cU$, hence
$\mathfrak A$ is contained in $\cK \oplus \cH$. But observe also that $A_0$ is contained in the nilradical of $G$,
hence $\mathfrak A$ is contained in $(\cK \oplus \cH) \cap \cN = (\cK \cap \cN) \oplus (\cH \cap \cN) = \cH \cap \cN$ (cf. Lemmas~\ref{le:KpasdansN}
and \ref{le:spectre}). In other words, the $T_0$-orbits are tori contained in the orbits of the Anosov action.

The intersection $A_0 \cap \Gamma$ is a lattice in $A_0$. It follows that the projection $\Gamma^\sharp$ of $\Gamma$ in $G^\sharp$ is a lattice.
Let $K^\sharp$ be the projection of $K$ in $G^\sharp$, and $\cH^\sharp$ the projection of $\cH$ in the Lie algebra $\cG^\sharp$ of $G^\sharp$. Clearly,
$(G, K, \Gamma, \cH)$ is a central extension over $(G^\sharp, K^\sharp, \Gamma^\sharp, \cH^\sharp)$.

The center of $G^\sharp$ is the quotient $Z^\sharp=Z/A$. If we reproduce the initial study of the induced action on $\Gamma^\sharp\backslash G^\sharp$,
we now obtain a \textit{finite} group of isometries of $\Gamma^\sharp\backslash G^\sharp$. It means precisely that $Z^\sharp \cap \Gamma^\sharp$ has finite index
in $Z^\sharp$. The Lemma is proved by taking the quotient $G'=G^\sharp/Z^\sharp$.

\end{proof}

%%%%%%%%%%%%%%

\subsection{Solvable case}
\label{sub:solvablecase}

\begin{theorem}
\label{th:solvable}
Every algebraic Anosov action $(G, K, \Gamma, \cH)$ where $G$ is a solvable Lie group is
commensurable to a nil-suspension over the suspension of an Anosov $\ZZ^{p}$-action on the torus ($p \in \NN^*$).

In particular, up to commensurability, the compact Lie group $K$ is finite (hence trivial if connected).
%and the algebraic Anosov action is obtained from the suspension of an Anosov action of $\ZZ^{k'}$ on the torus by taking successive
%abelian extensions and hyperbolic $\TT^\ell$-suspensions.
\end{theorem}

This \S~is devoted to the proof of Theorem~\ref{th:solvable}.

Recall that we are assuming, up to commensurability, that $K$ is connected.
Hence, since $G$ is solvable, the compact subgroup $K$ is a torus, and $\cH \oplus \cK$ is nilpotent.
Let $\cN$ denotes as usual the nilradical of $\cG$. Then $[\cG, \cG] \subset \cN$. Let $\bar{\cG}$ be the quotient algebra $\cG/\cN$. It is abelian,
$\approx \RR^p$ for some positive integer $p$.

According to Lemma~\ref{le:commutator}, $\cS$ and $\cU$ are contained in $\cN$. Therefore, the projections of $\cS$ and $\cU$ in $\bar{\cG}$
are trivial, and $\bar{\cG} = \bar{\cH} \oplus \bar{\cK}$, where $\bar{\cH}$ and $\bar{\cK}$ denote the projections of $\cH$ and $\cK$ in $\bar{\cG}$.

Similarly, the abelian Lie group $\bar{G} = G/N$
is the product $\bar{H} \times \bar{K}$ of the projections of $H$ and $K$ by $p: G \to G/N$. Using item $(4)$ of the definition of commensurability,
one can change $\cH$ by the graph in $\cK \oplus \cH$ of some linear map $\cH \to \cK$ so that now the intersection
$\bar{H} \cap \bar{\Gamma}$ is a lattice in $\bar{H}$. Then, the inverse image $p^{-1}(\bar{H})$ is
a normal subgroup $G'$ such that $G/G'$ is compact. In particular, $\Gamma \cap G'$ has finite index in $\Gamma$, ie. one
can assume up to commensurability that $\Gamma$ is contained in $G'$. Moreover, since $\bar{K}= p(K)$ is
supplementary to $\bar{H}$ in $\bar{G}$, the group $G$ is generated by $G'$ and $K$.
Then, by item $(3)$ of the definition of commensurability, $(G, K, \Gamma, \cH)$ is commensurable to $(G', K \cap G', \Gamma, \cH)$.
But, by construction of $G'$, the projection in $G'/N$ of $K \cap G'$ is trivial; therefore,
according to Lemma~\ref{le:KpasdansN}, $K \cap G'$ is finite. It follows that, up to commensurability, one can assume that
$K$ is trivial, ie. that the action is simplified.

We have an exact sequence:
$$0 \to \cN \to \cG \to \bar{\cH} \to 0$$
%This sequence is split since $\cH$ is a subalgebra of $\cG$.

Since $G$ can be assumed simply-connected (cf. Lemma~\ref{le:simplyconnected})
the nilradical $N$ and the abelian factor $\bar{G}$ are simply-connected.
%and diffeomorphic to Euclidean spaces. Hence $G$ is diffeomorphic
%to $\RR^{dimG}$: $G$ is contractible, and $\Gamma\backslash G$ is a $K(\Gamma, 1)$.
Hence, we also have a %split
exact sequence:
\begin{equation*}
\label{eq:exx}
1 \to N \to G \to \bar{H} \to 1
\end{equation*}

Let $\Lambda$ be the intersection $N \cap \Gamma$ and $\bar{\Gamma}$ the image $p(\Gamma)$. According to Theorem~\ref{thm:latticesol}
$\Lambda$ and $\bar{\Gamma}$
are lattices in $N$, $\bar{G}$, respectively. In particular, $\bar{\Gamma} \approx \ZZ^p$.

Let $N_1 = [ N, N ]$ be the first term of the lower central serie of $N$, and let $G_1 := G/N_1$. According to Theorem~\ref{th:latticenil} the intersection $\Lambda \cap N_1$ is a lattice $\Lambda_1$ in $N_1$.
Consider the subalgebra $\cN_H := \cN_1 + (\cH \cap \cN)$. It is contained in $\cN$, hence nilpotent. Moreover, we claim that
it is an ideal in $\cG$. Indeed, let $\op x = \op y + \op z$ be an element of $\cN_H$ with $\op y \in \cN_1$ and $\op z \in \cH \cap \cN$.
Since $\cS \oplus \cU \subset \cN$ and $\op x \in \cN$,
we clearly have ${\rm ad}{\op x}(\cS \oplus \cU) \subset [\cN, \cN] = \cN_1 \subset \cN_H$.
Now ${\rm ad}{\op y}(\cH) \subset [\cN_1, \cG] \subset \cN_1 \subset \cN_H$ and
${\rm ad}{\op z}(\cH) \subset [\cN, \cH] \cap [\cH, \cH] \subset \cN \cap \cH \subset \cN_H$,
hence ${\rm ad}{\op x}(\cH) \subset \cN_H$. Since $\cG = \cS \oplus \cU \oplus \cH$, our claim follows.

Let $N_H$ be the normal subgroup tangent to $\cN_H$. Since $\Gamma \cap N_1$ is a lattice in $N_1$ and $\Gamma \cap (N \cap H)$ is a lattice in $N \cap H$,
it follows that $\Gamma \cap N_H$ is a lattice in the nilpotent Lie group $N_H$. Hence, all the criteria for nil-suspensions are satisfied:
$(G, \Gamma, \cH)$ is a nil-suspension over the algebraic action $(G', \Gamma', \cH')$, where $G'=G/N_H$, $\cH' = \cH/(\cH \cap \cN_H) = \cH/(\cH \cap \cN)$
and $\Gamma'$ is the projection of $\Gamma$ in $G'$ (observe that $\Gamma'$ is a uniform lattice since $\Gamma \cap N_H$ is a lattice in $N_H$).

Let $\cG'$ be the Lie algebra of $G'$: its nilradical $\cN' = \cN/\cN_H$ (which is abelian) has now a trivial intersection with $\cH'$, hence
$\cN' = \cS' \oplus \cU'$. We still have $\cG'/\cN' \approx \cG/\cN \approx \bar{\cH}$.

It follows that $\cH'$ is abelian, and the exact sequence:
$$0 \to \cU' \oplus \cS' \to \cG' \to \bar{\cH} \to 0$$
is split. Therefore, $G'$ is a semi-direct product $N' \rtimes H'$. Once more, $\Lambda' := \Gamma' \cap N'$ and $\Delta' = \Gamma' \cap H'$ are lattices
in respectively $N'$, $H'$. The semi-direct product $\Delta' \ltimes \Lambda'$ is a lattice of $G'$ contained in $\Gamma'$. Up to commensurability,
one can assume the equality $\Gamma' = \Delta' \ltimes \Lambda'$.

Geometrically, this result means that the action of $H'$ is a suspension of an action of $\Delta' \approx \ZZ^p$
on the torus $\Lambda'\backslash N'$. This action is Anosov since the action of an Anosov element $h_0$ in $\Delta'$
is hyperbolic.

Theorem~\ref{th:solvable} is proved.

\subsection{Semisimple case}
\label{sub:sscase}
In this section, we assume that $G$ is semisimple.

\begin{theorem}
\label{th:ss}
Every algebraic Anosov action $(G, K, \Gamma, \cH)$ where $G$ is a semisimple Lie group is commensurable to
a modified Weyl chamber action.
\end{theorem}

The section is entirely devoted to the proof of Theorem~\ref{th:ss}.

According to Lemma~\ref{le:centrefini} in the appendix, we can assume, dividing by the intersection between $\Gamma$ and the center of $G$,
that the center of $G$ is finite. In particular, $G$ is a finite covering of its own adjoint group, and we can define (modified) Weyl chamber actions on $G$.
Moreover, $G$ is "nearly" algebraic: we can decompose every element of $\cG$ as the sum of its
nilpotent, elliptic and hyperbolic parts (cf. \S~\ref{sub:ss}).

According to Proposition~\ref{pro:CSA}, the sum $\cH \oplus \cT_1 \oplus \fa_0$ where $\cT_1$ is the radical of $\cK$
and $\fa_0$ a maximal abelian subalgebra of $\cK^\ast$ is a CSA, that we denote by $\cA$. In particular, $\cH \oplus \cT_1 \oplus \fa_0$
is abelian, and every element of $\cH$
is semisimple (as compact subalgebras, $\cT_1$ and $\fa_0$ are elliptic).

Let now $\cH_{hyp}$ be the abelian algebra comprising hyperbolic components of $\cH$ (hence of $\cA$), and let $\fa$ be a Cartan subspace containing $\cH_{hyp}$. The centralizer ${\mathcal Z}(\fa)$ of $\fa$ is $\fa \oplus \cK_0$, where
$\cK_0$ is the Lie algebra of a compact subgroup $K_0$ of $G$ (cf. \S~\ref{sub:ss}).
The following Lemma is quite classical, and is an ingredient of the proof that CSA's are characterized by the maximal dimension
of their intersections with Cartan subspaces.

\begin{lemma}
\label{le:chcsa}
$\cH \oplus \cK = \fa \oplus \cK_0$
\end{lemma}

\begin{proof}

Let $\fh_0$ be an Anosov element of $\cH$ and let $\fh_0^{hyp}$ be the hyperbolic part of $\fh_0$.
The splitting $\cG = \cH \oplus \cK \oplus \cS \oplus \cU$ is ${\rm ad}(\fh_0^{hyp})$-invariant, where $\cS$ is ${\rm ad}(\fh_0^{hyp})$-contracted,
$\cU$ is ${\rm ad}(\fh_0^{hyp})$-expanded, and $\cH \oplus \cK$ is the space of generalized $0$-eigenvectors of ${\rm ad}(\fh_0^{hyp})$.
Since ${\rm ad}(\fh_0^{hyp})$ is hyperbolic,  $\cH \oplus \cK$ is actually the kernel of ${\rm ad}(\fh_0^{hyp})$.

On the other hand, $\fh_0^{hyp}$ lies in $\cH_{hyp}$, hence in $\fa$. Therefore, the centralizer of $\fa$ is contained in the centralizer of $\fh_0^{hyp}$:
\begin{equation}
\label{eq:<}
\fa \oplus \cK_0 \subset \cH \oplus \cK
\end{equation}

In particular, the adjoint action of $\fa$ preserves the decomposition $\cH \oplus \cK$.
This action is trivial on the factor $\cH$ (since $\cH$ is abelian and commutes with every element of $\cK$); and on the factor $\cK$, it is $\RR$-split.
However, $\cK$ is the Lie algebra of a compact Lie group, hence the eigenvalues of any automorphism
on it is purely imaginary. It follows that the action of $\fa$ is trivial on $\cH \oplus \cK$.
Hence, $\cH \oplus \cK$ is contained in the centralizer ${\mathcal Z}(\fa)=\fa \oplus \cK_0$ of $\fa$.
The lemma follows.
\end{proof}

The algebra $\cH \oplus \cK = \fa \oplus \cK_0 = \cZ(\fa)$ admits a (unique) Levi decomposition $(\cH \oplus \cT_1) \oplus \cK^\ast = (\fa \oplus \cT_e) \oplus \cK^\ast_0$. Elements in $\cK$ are elliptic, hence have no component in $\fa$:
$\cK$ is a direct sum $\cK^\ast \oplus \cK_e$ where $\cK_e$ is a subspace of $\cT_e$.
Since $\cH$ is contained in the center of $\cZ(\fa)$, it is contained in $\fa \oplus \cT_e$, and it is supplementary to $\cK_e$ therein.
Consider $\cH' = \fa \oplus (\cH \cap \cT_e)$: it contains $\fa$, and clearly the modified Weyl chamber action
$(G, K=K^\ast.K_e, \Gamma, \cH')$ is equivalent to $(G, K, \Gamma, \cH)$
since $\cH' \oplus \cK = \cH \oplus \cK$. Theorem~\ref{th:ss} is proved.

%Therefore, the classification of algebraic Anosov actions of $\RR^k$ when $G$ is semisimple reduces to the classification of
%decompositions $\cH \oplus \cK$ of the centralizer of a Cartan subspace, where $\cH$ is abelian, and $\cK$ is tangent to a compact subgroup.

\begin{remark}
\label{rk:CSAhyperbolic}
In particular, $\cH$ contains the Cartan subspace $\fa$. Therefore, the CSA mentioned in Proposition~\ref{pro:CSA} is hyperbolic
in the sense of \S\ref{sub:ss}.
\end{remark}

\subsection{Mixed case}
\label{sub:mixedcase}
Now we study an algebraic Anosov action $(G, K, \Gamma, \cH)$ such that $G$ has a non trivial Levi factor $L$, and a non-trivial
(solvable) radical $R$.

\begin{theorem}
\label{th:mixed}
Let $(G, K, \Gamma, \cH)$ be an algebraic Anosov action, where $G$ is not solvable and not semisimple.
Then:

-- either $(G, K, \Gamma, \cH)$ is commensurable to an algebraic Anosov action on a solvable Lie group (it happens when $L$ is compact),

-- either $(G, K, \Gamma, \cH)$ is commensurable to a central extension over a (modified) Weyl chamber action (it happens when $G$ is reductive),

-- or $(G, K, \Gamma, \cH)$ is commensurable to a nil-suspension over an algebraic Anosov action which is commensurable to
a reductive algebraic Anosov action, ie. a central extension over a (modified) Weyl chamber action.
\end{theorem}

\begin{remark}
\label{rk:pastop}
The last part of the statement of Theorem~\ref{th:mixed} (the case of non-compact Levi factor) may appear unnecessarily complicated;
one could think that it could be replaced by the simplified statement:

    \textsf{$(G, K, \Gamma, \cH)$ is commensurable to a nil-suspension over a central extension over a (modified) Weyl chamber action.}

Actually it is not true. The point is that, as shown by the example in \S~\ref{sub:nilcentral} if $(G_1, K_1, \Gamma_1, \cH_1)$, $(G_2, K_2, \Gamma_2, \cH_2)$ are two commensurable algebraic Anosov actions, it could be possible to construct nil-supensions (even abelian hyperbolic)
over $(G_1, K_1, \Gamma_1, \cH_1)$ that are not commensurable to any nil-suspension over $(G_2, K_2, \Gamma_2, \cH_2)$.
Remark~\ref{sub:simplecompact} is another evidence that we have to take care in the formulation.
\end{remark}

Under the hypothesis of this Theorem, $R$ and $L$ are non-trivial.
Recall that $\pi: \tilde{G} \to G$, $\tilde{R}$ denote the universal coverings. We identify $\tilde{G}/\tilde{R}$ with a Levi factor,
so that $\tilde{G}$ is a semidirect product $\tilde{L} \ltimes \tilde{R}$.

\begin{proposition}
\label{pro:pascompact}
For any simple compact factor $C$ of $\tilde{L}$, the adjoint representation ${\rm Ad}: C \to {\rm GL}(\cR)$ is not trivial.
\end{proposition}

\begin{proof}
Assume not. Then, since $C$ commutes with the other simple factors of $\tilde{L}$, its Lie algebra $\cC$
is an ideal of $\cG$. In particular, it is preserved by the adjoint action ${\rm ad}(\fh_0)$ of any Anosov element.
Let $\cC = \bar{\cH} \oplus \bar{\cK} \oplus \bar{\cS} \oplus \bar{\cU}$ be the restriction to $\cC$ of the spectral decomposition
of ${\rm ad}(\fh_0)$ (cf. Lemma~\ref{le:spectre}). Since $C$ is compact, this restriction is elliptic.
Hence elements of $\cC$ have no components in $\cS \oplus \cU$, hence $\cC$ is contained in $\cH \oplus \cK$. Let $q$ be the restriction to
$\cC$ of the projection from $\cK \oplus \cH$ onto $(\cK \oplus \cH)/\cK \approx \cH$: since $C$ is simple, and  not abelian,
this restriction is trivial: $\cC$ is actually contained in $\cK$. This contradicts Remark~\ref{rk:noidealK}.
\end{proof}

Therefore, one can use Theorem~\ref{thm:latticeR} in the appendix: $\Gamma \cap R$ and ${\Gamma} \cap N$ are lattices in respectively
$R$, $N$. Moreover, the projections of ${\Gamma}$ in $G/R$ and the reductive group $G/N$ are uniform lattices.
From now, we distinguish two cases:

\subsubsection{Some Anosov element $\fh_0$ lies in the radical $\cR$:} The restriction of ${\rm ad}(\fh_0)$ to $\cS$
is an automorphism, hence every element of $\cS$ has the form ${\rm ad}(\fh_0)(s) = [\fh_0, s]$ where $s$ lies in $\cS$.
Since $\cR$ is an ideal, it follows that $\cS$ is contained in $\cR$.

Similarly, $\cU$ is contained in $\cR$. Therefore, $\cG/\cR$ is the projection
of $\cH \oplus \cK$. Since $\cG/\cR$ is semisimple, it follows that $\cG/\cR$ is compact and that $\cH$ is entirely contained in $\cR$.
Since the projection of ${\Gamma}$ in $G/R$
is discrete, it is actually finite: up to a commensurability (item $(1)$ of the definition) we have $\Gamma \subset R$.

It follows from item $(3)$ of the definition of commensurability that $(G, K, \Gamma, \cH)$
is commensurable to $(R, K \cap R, \Gamma, \cH)$: we are reduced to the solvable case.
Theorem~\ref{th:mixed} in this case follows from Theorem~\ref{th:solvable}.

\subsubsection{No Anosov element of $\cH$ lies in $\cR$:}
Up to a nil-suspension with principal fibers $(\Gamma \cap N)\backslash N$, we can assume that
the niradical of $G$ is trivial, ie. that $G$ is reductive\footnote{We use this observation only now, because it was unnecessary
in the previous case, in which we can avoid the loss of precision in the Theorem as explained in Remark~\ref{rk:pastop}. Of course, this observation
is also unnecessary if we already know that $G$ is reductive!}.

We have an exact sequence:
\begin{equation}
\label{eq:reduc}
1 \to R \to G \to G/R \to 1
\end{equation}
This exact sequence is central (cf. \cite[Chapter 2, Theorem~5.1]{onishchik}).
Hence $\cR \cap \cK$ is an ideal in $\cG$. According to Remark~\ref{rk:noidealK}, we can assume, up to commensurability, that this intersection
is trivial. Let $\fh_0$ be an Anosov element of $\cH$. Since $\cR$ is the direct of its intersections with $\cS$, $\cU$ and $\cH$ (its intersection between $\cK$ being trivial),
and since $\cR$ lies in the center of $\cG$, we obtain that $\cR$ is contained in $\cH$.

Therefore, the action by translation by $R$ defines a central extension over an algebraic Anosov action $(G', K', \Gamma', \cH')$
where $G'$ is now semisimple. Since we have already solved the semisimple case, it seems at first glance that we have proved
Theorem~\ref{th:mixed}, but it is not true. Indeed, in this way, we have replaced three intermediate Anosov actions under study by commensurable
ones: at the initial algebraic Anosov action (in order to use the "Preliminary results" of \S~\ref{sub:commen} and \ref{sub:prelim}); at
the induced algebraic action on the reductive group $G/N$ (when we use Remark~\ref{rk:noidealK} in order to claim that $\cR \cap \cK$ is trivial);
and finally to the semisimple factor $G/R$ when we use Theorem~\ref{th:ss}.

In order to really achieve the proof of Theorem~\ref{th:mixed} we need to perform the replacement by a commensurable action
only to $G/N$, ie. to prove the following particular case:

\begin{proposition}
\label{pro:reductive}
Any algebraic Anosov action $(G, K, \Gamma, \cH)$ where $G$ is reductive is commensurable to a central extension over
a modified Weyl chamber action.
\end{proposition}

Let $\bar{\fh}_0$ be the projection in $\cL = \cG/\cR$ of the Anosov element $\fh_0$.
The adjoint action of $\bar{\fh}_0$ on $\cL$ coincide with the action induced by ${\rm ad}(\fh_0)$, hence admits as spectral
decomposition $\cL = \bar{\cH} \oplus \bar{\cK} \oplus \bar{\cS} \oplus \bar{\cU}$ the projection of the spectral decomposition $\cG = \cH \oplus \cK \oplus \cS \oplus \cU$ of ${\rm ad}(\fh_0)$.
If $\bar{\cS}$ is trivial, then the same is true for $\bar{\cU}$, since the trace of ${\rm ad}(\fh_0)$ is $0$. Then $\cG$ reduces to $\bar{\cH} \oplus \bar{\cK}$, with $\bar{\cH} \neq \{ 0\}$ (since it contains $\bar{\fh}_0$). This is impossible since $\cL$ is semisimple.

Therefore, $\bar{\cS}$ and $\bar{\cU}$ are non-trivial, and $(G/R, \bar{K}, \bar{\Gamma}, \bar{\cH})$ is an algebraic Anosov action.
The key point is that in the proof of Theorem~\ref{th:ss} we have used Lemma~\ref{le:centrefini}: we have divided $G/R$ by the intersection
between $\bar{\Gamma}$ and the center of $G/R$.

However, we can reproduce word-by-word the proof in the semisimple
case. We obtain, after maybe replacing $\cH$ by another nilpotent subalgebra $\cH'$ like in item ($4$) of the definition of commensurability,
that $\bar{\cH} = \fa \oplus \cH_e$ and $\bar{\cK} = \cK_0^\ast \oplus \cK_e$, where $\fa$ is a Cartan subspace,
$\cK_0^\ast$ the semisimple part of $\cZ(\fa)$, and $\cH_e \oplus \cK_e$ a splitting of the radical of the compact part $\cK_0$ of $\cZ(\fa)$.
The little difference is that the Lie subgroup $\tilde{K}_0$ of $G/R$ tangent to $\cK_0$ might be non-compact in $G/R$, since $G/R$
might have infinite center as long as we do not use Lemma~\ref{le:centrefini}.

In order to conclude, we have to show that the following:

\begin{center}
\textit{Claim: $\Gamma$ contains a subgroup $Z_\Gamma$ that has finite index in the center $Z_G$ of $G$.}
\end{center}

Indeed, if we prove the claim, the Levi factor of $(G/Z_\Gamma)/(R/Z_\Gamma)$ has finite center.
Hence $(G, K, \Gamma, \cH)$ is commensurable to an algebraic Anosov action whose Levi factor has finite center,
and we conclude as above.

We prove the claim as follows: let us work in the universal covering,
hence consider the \textit{generalized} algebraic Anosov action $(\tilde{G}, \tilde{K}, \tilde{\Gamma}, \cH)$ (cf. Definition~\ref{def:algebraicaction2}).
Then $\tilde{G} = \tilde{L} \times \tilde{R}$. Pay attention to the fact that the semidirect product is a product. If the intersection
$\tilde{\Gamma} \cap \tilde{L}$ was a lattice in $\tilde{L}$, we would easily prove the claim by applying Lemma~\ref{le:centrefini}. The way
how this property fails is expressed by the projection of the morphism $\tilde{\Gamma} \to \tilde{R}$ defined by the projection on
the second factor. We will decompose this morphism in several pieces that we will study one-by-one.

First of all, observe that according to Proposition~\ref{pro:pascompact}, $\tilde{L}$ has no compact factors.
Let $\bar{\Gamma}$ be the uniform lattice in $\tilde{L}$, obtained by identifying the Levi factor with $\tilde{G}/\tilde{R}$
(hence it is not the intersection $\tilde{\Gamma} \cap \tilde{L}$).

Let $\tilde{L} = \tilde{F}_1 \times ... \times \tilde{F}_r$ the decomposition of $\tilde{L}$ in factors so that $\bar{\Gamma}$
is the product $\bar{\Gamma}_1 \times ... \times \bar{\Gamma}_r$ where each $\bar{\Gamma}_i$ is a irreducible uniform lattice in $\tilde{F}_i$.
%For every $i$, let $U_i$ the subgroup of $\tilde{G} = \tilde{L} \times \tilde{R}$ comprising elements whose projection in
%$\tilde{L}_j$ are trivial except for $j=i$: the projection of $U_i$ in $\tilde{\Gamma}\backslash\tilde{G}$ is a closed submanifold,
%hence compact, hence $\tilde{\Gamma}_i = U_i \cap \tilde{\Gamma}$ is a uniform lattice in $U_i$.
It defines a similar splitting $\tilde{\Gamma} = \tilde{\Gamma}_1 \times ... \times \tilde{\Gamma}_r$ where each $\tilde{\Gamma_i}$
is an uniform lattice in $\tilde{F}_i \times \tilde{R}$.
For each $i$, consider the projection $\alpha_i: \tilde{\Gamma}_i \to \tilde{R}$ on the second factor. We have an exact central sequence:
$$1 \to \Lambda \to \tilde{\Gamma}_i \to \bar{\Gamma}_i \to 1$$
where $\Lambda$ is the lattice $\tilde{\Gamma} \cap \tilde{R}$ of $\tilde{R}$. Consider the induced morphism
from $\bar{\Gamma}_i$ into $\tilde{R}/\Lambda$.
Since $\tilde{R}/\Lambda$ is abelian, it induces a morphism:
$$\bar{\alpha}_i: \bar{\Gamma}_i/[\bar{\Gamma}_i, \bar{\Gamma}_i] \to \tilde{R}/\Lambda$$

There are several possibilities:

\begin{itemize}
  \item \textit{If $\tilde{F}_i$ is not locally isomorphic to $\op{SU}(1,n)$ or $\op{SO}(1,n)$:} then, according to \cite[Theorem 7.1, page 98]{vinberg2},
  $\bar{\Gamma}_i/[\bar{\Gamma}_i, \bar{\Gamma}_i]$ is finite (because then $\tilde{F}_i$ and its lattices have property (T)). By restricting to a finite index subgroup, we can assume that $\bar{\alpha}_i$
  is trivial. It means that $\tilde{\Gamma}_i$ preserves $\tilde{F}_i \times \Lambda$. Therefore, $\tilde{\Gamma}_i \cap \tilde{F}_i$  is an uniform
  lattice in $\tilde{F}_i$, hence contains a finite index subgroup $Z_i$ of the center of $\tilde{F}_i$ according to Lemma~\ref{le:centrefini}.
  \item \textit{If $\tilde{F}_i$ is locally isomorphic to $\op{SU}(1,n)$ with $n\geq2$, or to $\op{SO}(1,k)$ with $k \geq 3$:} Then, the center of $\tilde{F}_i$ is finite. Define $Z_i$ as the trivial group.
  \item \textit{If $\tilde{F}_i$ is locally isomorphic to $\op{SU}(1,1) \approx \op{SO}(1,2) \approx \op{PSL}(2)$:} Then either $\tilde{F}_i$ has finite center and we define
  $Z_i$ has the trivial group, or $\tilde{F}_i$ is the universal covering $\widetilde{\op{SL}}(2)$ of $\op{PSL}(2)$. In the second case, we define $Z_i$ as the intersection between $[\Gamma_i, \Gamma_i]$,
  and the center of $\widetilde{\op{SL}}(2)$. According to Lemma~\ref{le:latticesl2}, $Z_i$ has a finite index intersection with the center of $\widetilde{\op{SL}}(2)$. But, since $Z_i$ is contained in
  $[\Gamma_i, \Gamma_i]$, it is contained in the first commutator group of $\tilde{R} \times F_i$,
  ie., $F_i$.
\end{itemize}

Now take the product $Z_1 \times ... \times Z_r$ of every subgroup constructed in each $\tilde{F}_i$. It is a discrete subgroup of $\tilde{\Gamma}$
which has finite index in the center of $\tilde{L} = \tilde{F}_1 \times ... \times \tilde{F}_r$ (since this center is the product of the centers of the $\tilde{F}_i$'s). The Claim is proved, and the proof of
Theorem~\ref{th:mixed} is complete.\qed

\section{Conclusion}
\label{sec:conclusion}

\subsection{Algebraic Anosov flows}
Our study covers the special case already treated by P. Tomter; the case of algebraic Anosov flows, ie. the case
where the acting group $H$ is $\RR$. Actually, most of the difficulties appearing in the proofs of
Theorems \ref{th:solvable} and \ref{th:mixed} are greatly simplified in this case.
For example, assume that the Levi factor $L$ of $G$ is not trivial. Then, $L$ must be simple and of rank one, and the underlying
Weyl chamber flow cannot be a modified one. This Levi factor must obviously be of real rank one. Moreover,
according to the proof of Theorem~\ref{th:mixed}, $R/N$ is contained in the acting group $H=\RR$, and therefore must be trivial
since $H$ already contains a Cartan subspace of $G/R$.

In the case where $L$ is trivial, ie. when $G$ is solvable, according to Theorem~\ref{th:solvable}, $(G, K, \Gamma, \RR)$ is commensurable to a nil-suspension over the suspension of an Anosov action of $\ZZ^k$ on a torus . Obviously $k=1$, and the nil-suspension must be hyperbolic.

Hence we recover following result, with a proof somewhat simpler than the one in \cite{tomter1}, \cite{tomter2}:

\begin{theorem}
\label{thm:tomter}
Let $(G, K, \Gamma, \RR)$ be an algebraic Anosov flow. Then, $(G, K, \Gamma, \RR)$ is commensurable to either
the suspension of an Anosov automorphism of a nilmanifold, or to a hyperbolic nil-suspension over the geodesic flow
of a locally symmetric space of rank one.\qed
\end{theorem}

\subsection{Algebraic Anosov actions of codimension one}
Recall that an Anosov action is said \textit{of codimension one} if there is an Anosov element of $\cH$ for which the unstable subalgebra $\cU$ has dimension one.
It is conjectured (cf. \cite{Bar-Maq}) that any algebraic Anosov action of $\RR^k$ of codimension one is
a $\TT^\ell$-extension over the suspension of an Anosov action of $\ZZ^k$ on a torus, or over the geodesic flow
of the unit tangent bundle of a closed hyperbolic surface. A first step is to check this
conjecture in the case of algebraic Anosov action.

\begin{proposition}
\label{pro:codimone}
Let $(G, K, \Gamma, \cH)$ be an algebraic Anosov action of codimension one.
Then, $(G, K, \Gamma, \cH)$ is commensurable to either
a nil-extension over a codimension one Anosov action of $\ZZ^k$ on a torus, or to a nil-extension over the geodesic flow of the unit tangent bundle of a closed hyperbolic surface.
\end{proposition}

\begin{proof}
Let us first consider the case where $G$ is solvable: according to Theorem \ref{th:solvable},
the action is commensurable to a nil-suspension over an Anosov action of $\ZZ^p$ on a torus.
Clearly, this Anosov action of $\ZZ^p$ must be of codimension one, and the nil-suspension has to
be a nil-extension since if not it would increase the dimension of $\cU$.

When $G$ is not solvable, $(G, K, \Gamma, \cH)$ is commensurable to a nil-suspension
over an algebraic Anosov action $(G', K', \Gamma', \cH')$ where $G'$ is the reductive quotient $G/N$
of $G$. As in the solvable case, this nil-suspension must be a nil-extension. It follows that
the (solvable) radical $R$ of $G$ must be contained in $H$, hence be nilpotent: we have $R=N$,
hence $G'$ is semisimple. Therefore, $(G', K', \Gamma', \cH')$ is a modified Weyl chamber action
of codimension one.

Now we observe that for modified Weyl chamber actions, $\cU$ and $\cS$ have the same dimension for every Anosov element (ie. every element of the Weyl chamber). Hence in our case they are
both one dimensional, meaning that $G'$ has only two roots. The classification of
semisimple Lie algebra implies that $G'$ must be isogenic to $\op{PSL}(2,\RR)$. The proposition follows.
\end{proof}

\subsection{A remark on equivalence between algebraic Anosov actions and hyperbolic Cartan subalgebras}
\label{sub:csach}
According to Proposition~\ref{pro:CSA}, we have observed that algebraic Anosov actions are closely related
to Cartan subalgebras. More precisely:

-- when $G$ is solvable, we can assume up to commensurability that $K$ is trivial. Then it follows immediately
from Proposition \ref{pro:CSA} that $\cH$ is a CSA of $\cG$.

-- when $G$ is non-solvable, we have proved that the projection of $\cH$ in the Levi factor $G/R$ must
contain a Cartan subspace $\fa$ (cf. Lemma~\ref{le:chcsa}). Therefore,
the CSA $\cH \oplus \fa_0 \oplus \cT_1$ as stated in Proposition~\ref{pro:CSA} is a hyperbolic CSA of $G$ (cf. \S~\ref{sub:CSA}).

Hence, according to Theorems \ref{thm:csasolvable} and \ref{thm:csageneral}, the CSA $\cH \oplus \fa_0 \oplus \cT_1$
does not depend (up to conjugacy in $G$) on the Anosov action.

%\subsection{Faithful actions}
%As we have already observed, a nil-extension is never a faithful action. The following proposition
%shows that it is essentially the only way for the action not to be faithful:

%\begin{proposition}
%An algebraic Anosov action $(G, K, \Gamma, \cH)$ is faithful if and only the center of $G$
%is trivial. In particular, any algebraic Anosov action $(G, K, \Gamma, \cH)$ is a commensurable to
%central extension over a faithful algebraic Anosov action $(G', K', \Gamma', \cH')$.
%\end{proposition}

%\begin{proof}
%Let $(G, K, \Gamma, \cH)$ be an algebraic Anosov action.
%According to Lemma \ref{le:sanscentre}, $(G, K, \Gamma, \cH)$ is commensurable to
%a central extension over an action $(G', K', \Gamma', \cH')$ such that the center of $G'$
%is trivial.

%Assume that $(G', K', \Gamma', \cH')$ is not faithful. Then, there is
%a non trivial element $h$ in $H'$ such that for every $g$ in $G'$, there is an element $k_g$ of $K'$ and an element $\gamma_g$ of $\Gamma$ such that:
%$$gh = \gamma_g gk_g$$

%If $G$ (hence $G'$) is solvable, we can assume up to commensurability that $K'$ is trivial, therefore
%$k_g$ is trivial. Then, $\gamma_g = ghg^{-1}$ varies continuously with $g$ but takes value in
%the discrete subgroup $\Gamma'$: it is a constant map. Therefore, $h = \gamma_g$ lies in the center
%of $G'$: contradiction.

%Therefore, $G'$ has a non-trivial Levi factor $L'$.

%\end{proof}

\section{Appendix}
\label{sec:appendice}

\subsection{Spectral decomposition of linear endomorphisms}
Let $V$ be a finite dimensional real vector space.
Let $f: V \to V$ be a linear endomorphism, and $V = E_1 \oplus ... \oplus E_k$ its spectral
decomposition in generalized eigenspaces: the restriction of the complexification $f^\CC$ to $E_i^\CC$
has only one complex eigenvalue $\lambda_i$.

Let $W \subset V$ be a $f$-invariant linear subspace, let
$p: V \to V/W$ be the quotient map and let $\bar{f}: V/W \to V/W$ be the induced linear map.

\begin{definition}
\label{def:hyperb}
The linear endomorphism $f: V \to V$ is hyperbolic if no complex eigenvalue of $f$
has vanishing real part. It is partially hyperbolic if the endomorphism induced on $V/{\rm Ker}f$ is hyperbolic.
\end{definition}

\begin{lemma}
\label{le:spectre}
The spectral decomposition of the restriction $f_{\mid W}$ is $W = (E_1 \cap W) \oplus ... \oplus (E_k \cap W)$,
and the spectral decomposition of $\bar{f}$ is $V/W = p(E_1) \oplus ... \oplus p(E_k)$.
\end{lemma}

\begin{corollary}
\label{cor:partialhyperb}
Let $W \subset V$ be $f$-invariant subspace. If $f$ is (partially) hyperbolic, then
$f_{\mid W}: W \to W$ and $\bar{f}: V/W \to V/W$ are (partially) hyperbolic.
\end{corollary}

\subsection{Lattices in  Lie groups}

\subsubsection{The nilpotent case}
\label{sub:latticenil}
Let $N$ be a simply connected nilpotent Lie group, and $\Lambda \subset N$ a uniform lattice.

Define inductively $N_{k+1} := [N, N_k]$. For some integer $n$, $N_{n+1}$ is trivial.

\begin{theorem}[Corollary 2 of Theorem 2.3 in \cite{raghunathan}]
\label{th:latticenil}
For every integer $k$, the intersection $\Lambda_k = \Lambda \cap N_k$ is a lattice in $N_k$.
\end{theorem}

\begin{remark}
\label{rk:latticenil}
Let $\bar{N}_k$ be the quotient $N_k/N_{k+1}$. It is isomorphic to $\RR^{p_k}$ for some integer $p_k$.
It follows from Theorem~\ref{th:latticenil} that the projection $\bar{\Lambda}_k$ of $\Lambda_k$
in $\bar{N}_k$ is a lattice. Hence $\bar{N}_k$ is naturally equipped with a structure over $\ZZ$,
which is preserved by every automorphism of $N$ preserving $\Lambda$. More precisely,
let $f: N \to N$ be such an automorphism, such that $f(\Lambda)=\Lambda$. Then
there are induced linear maps $f_k: \bar{N}_k \to \bar{N}_k$. Once fixed a basis of $\Lambda_k$,
this linear map is identified with a matrix in ${\rm GL}(p_k, \ZZ)$. In particular, the $1$-eigenspace and
the generalized $1$-eigenspace  of $f_k$ are rational subspaces, ie. their intersections with $\Lambda_k$ are lattices.
\end{remark}

\subsubsection{The solvable case}
Here we assume that $G$ is solvable. Then it contains a maximal normal nilpotent subgroup: the nilradical $N$.
The first commutator subgroup $[\Gamma, \Gamma]$ is nilpotent, hence contained in $N$.

\begin{theorem}[Corollary 3.5 in \cite{raghunathan}]
\label{thm:latticesol}
Let $\Gamma$ be a lattice in a connected solvable Lie group $G$. Then $\Gamma \cap N$ is a lattice in $N$,
and the projection of $\bar{\Gamma}$ in $G/N$ is a lattice.
\end{theorem}

\subsubsection{Semisimple Lie groups}
\label{sub:ss}
Let $G$ be a real semi-simple Lie group, and $\cG$ its Lie algebra. It splits has a product of simple factors
$G = F_1 \times ... \times F_m$ which is unique, up to permutation of the factors.

A crucial point is that $\cG$ is isomorphic to the Lie algebra of an algebraic subgroup (the adjoint of $G$)
in ${\rm GL}(n, \RR)$. In particular, every element $\mathfrak u$ of $\cG$ splits uniquely as a sum
${\mathfrak u} = {\mathfrak u}_n + {\mathfrak u}_{hyp} + {\mathfrak u}_e$ of two-by-two commuting elements
such that:

-- ${\mathfrak u}_n$ is nilpotent,

-- ${\mathfrak u}_{hyp}$ is hyperbolic, ie. the adjoint action of ${\mathfrak u}_{hyp}$ on $\cG$ is $\RR$-diagonalizable,

-- ${\mathfrak u}_e$ is elliptic, ie. ${\rm ad}({\mathfrak u}_e)$ is diagonalizable over $\CC$, and the eigenvalues have
all zero real part.

Every element commuting with ${\mathfrak u}$ commutes with each factor ${\mathfrak u}_n$, ${\mathfrak u}_{hyp}$, ${\mathfrak u}_e$.

${\mathfrak u}$ is \textit{semisimple} if its nilpotent part ${\mathfrak u}_n$ is $0$, ie. if
${\rm ad}({\mathfrak u})$ is $\CC$-diagonalizable.

\begin{definition}
\label{def:cartanalg}
A \textbf{Cartan subalgebra} is an abelian subalgebra of $\cG$ comprising semisimple elements
and maximal for this property. A Lie group tangent to a Cartan subalgebra is a \textbf{Cartan subgroup}.

A \textbf{Cartan subspace} is an abelian subalgebra of $\cG$ comprising hyperbolic elements
and maximal for this property. A Lie group tangent to a Cartan subspace is a \textbf{split Cartan subgroup}.
\end{definition}

Every Cartan subalgebra is equal to its own centralizer in $\cG$, whereas the centralizer of a Cartan subspace ${\mathfrak a}$
is always a sum ${\mathfrak a} + \cK$ where $\cK$ is a compact Lie algebra, meaning that $\cK$ is tangent in the adjoint group of $G$
to a compact subgroup $K$.

In $\cG$ there is only one Cartan subspace up to conjugacy, but, even if finite, the number
of conjugacy classes of Cartan subalgebras can be $>1$.

For example, in $\mathfrak{so}(1,3) \approx \mathfrak{sl}(2,\CC)$, there are two Cartan subalgebras not conjugate one to the other: one tangent
to the maximal abelian subgroup $\op{SO}_0(1,1) \times \op{SO}_0(2)$ of $\op{SO}_0(1,3)$, and the other tangent to the maximal abelian subgroup of diagonal matrices
in $\op{PSL}(2,\CC)$.

More precisely, Cartan subalgebras are characterized up
to conjugacy by the maximal dimension of its intersection with a Cartan subspace. In particular,
Cartan subalgebras containing an entire Cartan subspace are conjugate one to the other. They are of the form
${\mathfrak a} + {\mathcal T}$ where ${\mathcal T}$ is a maximal torus in the compact part $\cK$ of the centralizer of a Cartan subspace ${\mathfrak a}$.
We call them \textbf{hyperbolic Cartan subalgebras}.

Cartan subalgebras, hyperbolic or not, has all the same dimension, called the \textit{rank} of $G$. The common
dimension of Cartan subspaces is the \textit{$\RR$-rank} of $G$.

Note that the classical results we have stated above hold mainly in the case where $G$ is its own adjoint, ie. has trivial center.
It may happen that $G$ is some infinite covering over its adjoint group, and, for example, that the group $K$ tangent in $G$
to $\cK \subset \cZ(\fa)$ is an infinite covering
over a compact subgroup, therefore is not anymore compact. Moreover, results about lattices are more commonly stated
in the case when $G$ has a finite center. We can actually reduce to this case, thanks to the following lemma:

\begin{lemma}[Proposition 1.3, p. 68 in \cite{vinberg2}]
\label{le:centrefini}
Let $G$ be a real semisimple group without compact factors, $Z$ the center of $G$, and $\Gamma$ a uniform lattice. Then, $\Gamma \cap Z$ has a finite index in $Z$.\qed
\end{lemma}

%\begin{proof}
%Quite surprisingly for us, we were unable to find an explicit suitable reference in the litterature for this lemma, even if it seems
%fundamental for the study of lattices in general real semi-simple Lie groups.

%First observe that since the center of $\cG$ is trivial, $Z$ is a discrete subgroup.
%Equip $G$ with a left invariant metric. Left (or right!) translations induces an isometric action of $Z$ on $\Gamma\backslash G$, ie. an
%abelian subgroup of isometries of this action. The closure $T$ of this subgroup is compact (by Ascoli-Arzela Theorem), hence its identity component
%is a torus $T_0$ in the group of diffeomorphisms of $\Gamma\backslash G$. Moreover, every element of $T$ commutes with the action of $G$ on
%$\Gamma\backslash G$ by right translations, therefore $T_0$ is induced by \textit{left} translations by a connected abelian Lie subgroup
%$A$ of $G$. Now $A$ commutes with every element of $\Gamma$, hence its projection in the algebraic group $G/Z$ commutes with the Zariski
%closure of the projection of $\Gamma$, ie. the entire $G/Z$. It follows that $A$ is a connected Lie subgroup contained in the center
%of $G$, therefore, $A$ is trivial. It means that $T$ is a finite group, hence that the group of translations by $Z$ on $\Gamma\backslash G$ was already finite. The lemma follows.
%\end{proof}

We will also need the following folk fact:

\begin{lemma}
\label{le:latticesl2}
Let $\tilde{\Gamma}$ be a lattice in the universal covering $\widetilde{\op{SL}}(2,\RR)$. Then, the first commutator group $[\tilde{\Gamma}, \tilde{\Gamma}]$ contains a finite index subgroup of the center of $\widetilde{\op{SL}}(2,\RR)$.
\end{lemma}

\begin{proof}
Let $p: \widetilde{\op{SL}}(2,\RR) \to \op{PSL}(2,\RR)$ be the universal covering. The center $Z$ of $\op{PSL}(2,\RR)$ is an infinite cyclic subgroup.
The projection $\Gamma = p(\tilde{\Gamma})$ is a lattice in $\op{PSL}(2,\RR)$, hence a
a closed surface group; let say of genus $2$. It admits a presentation of the form:
$$\langle a_1, b_1, a_2, b_2 \mid [a_1, b_1][a_2, b_2]=1 \rangle$$

It follows that $\op{PSL}(2,\RR)$ has a presentation of the form:
$$\langle \tilde{a}_1, \tilde{b}_1, \tilde{a}_2, \tilde{b}_2, h \mid [h, a_i]=1, \; [h, b_i] =1, [\tilde{a}_1, \tilde{b}_1][\tilde{a}_2, \tilde{b}_2]= h^e \rangle$$

where $h$ lies in $Z$. The key fact is that the integer $e$ cannot vanish: it is related to the fact that
the unit tangent bundle of a closed surface of genus $g \geq 2$ is $2g-2$, hence nonzero.

For more details, see for example section $6$ in \cite{ghysolodov}.
\end{proof}

\subsubsection{Splitting of lattices in solvable/semisimple parts}

Consider now a general Lie group, which is not solvable, but also not semisimple: $\cG$
contains a normal solvable ideal. All these normal solvable ideals are contained in one of them,
the \textit{radical} $\cR$. The quotient algebra $\cG/\cR$ is semisimple.

\begin{theorem}[Levi-Malcev, see  Theorem~4.1 p. 18, and Theorem 4.3 p.20  in \cite{onishchik}]
\label{thm:levi}
The exact sequence $0 \to \cR \to \cG \to \cG/\cR \to 0$ splits, ie. $\cG$ contains a subalgebra
$\cL$, called \textbf{Levi factor}, which projects one-to-one over $\cG/\cR$.
Moreover, $\cL$ is unique up to conjugacy.
\end{theorem}

We also have:

\begin{theorem}[Corollary~3 in \S~ 4.3 of \cite{onishchik}]
\label{thm:simpleinLie}
Any semisimple subalgebra of $\cG$ is contained in a Levi factor of $\cG$.
\end{theorem}

From now, we consider the radical $R$ of $G$, tangent to $\cR$, a semisimple Lie subgroup $L$
(a Levi subgroup) tangent to a Levi factor. Observe that $L$ is not in general a Lie subgroup, but a "connected virtual subgroup";
 it might be dense in $G$ (cf. the Corollary and the example page 19 in \cite{onishchik}).
 We have $G= R.L$, meaning that every element $g$ of $G$ is a product $rl$, where
$r$ lies in $R$ and $l$ in $L$, and that the intersection $L \cap R$ is discrete.
The quotient $G/R$ is then isomorphic to the quotient of $L$ by a discrete central subgroup.

If $G$ is simply connected, this intersection is trivial, ie. the decomposition $g=rn$ is unique. More precisely,
$R$ and $L$ are simply connected Lie subgroups and $G = R \rtimes L$.

Let $N$ be the nilradical of $R$: it is
the maximal nilpotent normal subgroup of $G$. And
let $\Gamma$ be a uniform lattice in $G$.

It is not true in general that $\Gamma \cap R$
is a lattice in the radical $R$. Consider for example $G = \RR \times {\rm SO}(3, \RR)$: the radical
is $\RR$, but any subgroup generated by a non-trivial element is lattice, if this element has a torsion-free component in ${\rm SO}(3, \RR)$
then the lattice has trivial intersection with the radical.

However:

\begin{theorem}[Theorem\footnote{In \cite{vinberg2}, in a footnote page 107, it is observed that the proof in \cite{raghunathan} is incorrect. However, the proof has been from then completed, see \cite{wall} .}~{8.28} in \cite{raghunathan}]
\label{thm:latticeR}
Consider the adjoint action of the Levi factor $L$ on the radical $R$. Assume that the kernel of this action
contains no compact simple factor of $L$. Then, $\Gamma \cap R$ and $\Gamma \cap N$ are lattices in respectively
$R$, $N$. Moreover, the projections of $\Gamma$ in $G/R$ and the reductive group $G/N$ are lattices.
\end{theorem}

\subsection{Cartan subalgebras in general Lie algebras}
\label{sub:CSA}
Let $\cG$ be a real Lie algebra, not necessarily semisimple. For any subalgebra $\cA$ the normalizer $\cN(\cA)$ is:
$$\fN(\cA) := \{ \op x \;/\; [\op x, \op y] \in \cA \;\; \forall \op y \in \cA \}$$

\begin{definition}[cf. \cite{humphreys}, p. 80]
\label{def:csa}
A subalgebra $\cA$ of $\cG$ is a \textbf{Cartan subalgebra} (abrev. CSA) if is is nilpotent, and equal to its own normalizer.
\end{definition}

When $\cG$ is semisimple, this definition coincide with the one given in Definition~\ref{def:cartanalg} (see \cite[Corollary 15.3]{humphreys}).

\begin{proposition}[Lemma A in section 15.4 of \cite{humphreys}]
\label{pro:projetcsa}
Let $\phi: \cG \to \cG'$ be an epimorphism of Lie algebras. Then the image by $\phi$ of any CSA of $\cG$ is a CSA of $\cG'$.
\end{proposition}

There is another useful characterization of CSA:

\begin{definition}
\label{def:engel}
An \textbf{Engel subalgebra} of $\cG$ is a subalgebra of the form $L_0(\mathfrak h)$, where $\mathfrak h$ is an element of $\cG$ and  where $L_0(\mathfrak h)$ denotes the $0$-characteristic
subspace of the adjoint action of $\mathfrak h$ on $\cG$.
\end{definition}

\begin{lemma}[Theorem 15.3 in \cite{humphreys}]
\label{le:engel}
A subalgebra of $\cG$ is a CSA if and only if it is a minimal Engel subalgebra.
\end{lemma}

In particular, every Lie group contains a CSA. On the other hand, it is a classical result that if $\cG$ is the Lie algebra of an algebraic group over
an algebraically closed field, CSA in $\cG$ are all conjugate one to the other. It remains true when the field is not algebraically closed under some additional hypothesis:

\begin{theorem}[Theorem 16.2 in \cite{humphreys}]
\label{thm:csasolvable}
If $\cG$ is a solvable real\footnote{Actually, the ground field to be infinite is a sufficient condition.} Lie group, the Cartan subalgebras of $\cG$ are conjugate one to the other.
\end{theorem}

As we have already pointed out, CSA's are not unique up to conjugacy when the real Lie algebra is semisimple. However, the classification up to
conjugacy of CSA's reduces to the classification of maximal solvable subalgebras of $\cG$, ie. \textbf{Borel subalgebras}. Indeed, every CSA is contained in
a Borel subalgebra, and any CSA in a Borel subalgebra is a CSA in $\cG$. Now there is a natural 1-1 correspondence between Borel subalgebras
of $\cG$ and Borel subalgebras of the Levi factor $\cL$ (the correspondence is simply the projection). As a corollary, we get:

\begin{theorem}
\label{thm:csageneral}
Every Lie group $G$ contains a ~unique \textit{hyperbolic} CSA, ie. a CSA whose projection in the Levi factor is an hyperbolic CSA of $\cL$.
\end{theorem}

%%%$\mathscr{A B C D E F G H I J K L M N O P Q R S T U V X Y W Z}$
%%%$\mathcal{F} \  F   $


\begin{thebibliography}{10}

%\bibitem{anosov}
%D.V. Anosov.
%\newblock Roughness of geodesic flows on compact manifolds of negative
%curvature (in Russian).
%\newblock {\em Dokl.Akad.Nauk SSSR}, 145, 707--709, 1962.

\bibitem{Bartree}
T. Barbot.
\newblock Actions de groupes sur les $1$-vari\'et\'es non s\'epar\'ees et feuilletages de codimension un.
\newblock{\em Ann. Fac. Sciences Toulouse,} {\bf 7} (1998), 559--597.

\bibitem{Bar-Maq}
T. Barbot and C. Maquera.
\newblock Transitivity of codimension one Anosov actions of $\RR^k$ on closed manifolds.
{\em Ergod. Th. \& Dyn. Sys.,} {\bf 31} (2011), no. 1, 1--22.


\bibitem{BoLan} C. Bonatti, R. Langevin. Un exemple de flot d'Anosov
transitif transverse \`a un tore et non conjugu\'e \`a une suspension.
{\em Ergod. Th. \& Dyn. Sys.,} {\bf 14} (1994), 633--643.

\bibitem{brintop}
M.I.~Brin.
\newblock Topological transitivity of one class of dynamical systems, and
flows of frames on manifolds of negative curvature.
\newblock {\em Funcional. Anal. i Prilozen,} 9:9--19, 1975


\bibitem{bringroup}
M.I.~Brin.
\newblock The topology of group extensions of C-systems.
\newblock {\em Math. Zametki,} 3:453--465, 1975

\bibitem{fishkalininspatz}
D. Fisher, B. Kalinin, R. Spatzier.
 \newblock Global Rigidity of Higher Rank Anosov Actions on Tori and Nilmanifolds.
 \newblock {\em  arXiv:1110.0780 }.

\bibitem{franks}
 J.~Franks.
 \newblock Anosov diffeomorphisms.
\newblock {\em Global Analysis (Proc. Symp. Pure Math., 14)}, Amer. Math. Soc.,
 61--93, 1970.

 \bibitem{ghys}
 E. Ghys.
 \newblock Codimension one Anosov flows and suspensions.
 \newblock {\em Lecture Notes in Math. 1331}, Springer, Berlin,
59--72, 1988.

\bibitem{ghysolodov}
E. Ghys.
\newblock Groups acting on the circle.
\newblock {\em Enseign. Math. (2) 47},  329--407, 2001.

\bibitem{haereeb}
A. Haefliger, and G. Reeb.
\newblock Vari\'et\'es (non s\'epar\'ees) \`a une dimension et structures feuillet\'ees du plan.
\newblock {\em Enseign. Math. 3}, 107--125, 1957.

\bibitem{hector}
G. Hector and U. Hirsch.
\newblock Introduction to the geometry of the foliations. Part B
\newblock {\em Aspects of Mathematics, Friedr. Vieweg \& Sohn}, xi+234,1981.

 \bibitem{hirpush}
 M. Hirsch, C. Pugh and M. Shub.
 \newblock Invariant manifolds.
 \newblock {\em Lecture Notes 583, Springer Verlag},
 Berlin 1977.

 \bibitem{humphreys}
 J.E. Humphreys,
 \newblock Introduction to Lie algebras and representation theory (2nd edition)
 \newblock {\em Graduate Texts in Mathematics, Springer Verlag},
 Berlin 1972.

 \bibitem{Imhof}
 H.C. Im Hof.
 \newblock An Anosov action on the bundle of Weyl chambers.
{\em Ergod. Th. \& Dyn. Sys.,} {\bf 5} (1985), 587--593.

 \bibitem{kalininspatz}
 B. Kalinin, R. Spatzier.
 \newblock On the classification of Cartan actions.
 \newblock {\em G.A.F.A. Geom. funct. anal. (17)}, 468--490, 2007.

%\bibitem{katokihes} A. Katok, and R.J. Spatzier.
%First cohomology of Anosov actions of higher rank abelian groups and applications to rigidity.
%{\em Inst. Hautes \'Etudes Sci. Publ. Math. (79)}, 131--156, 1994.

\bibitem{KatSpatz1}
 A. Katok, R. Spatzier.
 \newblock First cohomology of Anosov actions of higher rank abelian groups and applications to rigidity.
 \newblock {\em Inst. Hautes \'etudes Sci. Publ. Math.}, 79:131--156, 1994.


 \bibitem{KatSpatz2}
 A. Katok, R. Spatzier.
 \newblock Differential rigidity of Anosov actions of higher rank abelian groups and algebraic lattice actions.
 \newblock {\em Proc. Steklov Inst. Math.}, no. 1 (216), 287--314, 1997.

\bibitem{KL2}
A. Katok, and J. Lewis,
\newblock Global rigidity results for lattice actions on tori and new examples of volume-preserving actions.
\newblock {\em Israel J. Math.( 93),} 253--280 1996.

\bibitem{manning}
A. Manning.
\newblock No new Anosov diffeomorphisms on Tori.
\newblock {\em Amer. J. Math.}, 96:422--429, 1974.


 \bibitem{matsumotosolv} S. Matsumoto. Codimension one foliations on
solvable manifolds,\textit{Comment. Math. Helv.,} {\bf 68} (1993), 633--652.

 \bibitem{newhouse}
 S.~E. Newhouse.
 \newblock On codimension one Anosov diffeomorphisms.
 \newblock {\em Amer. J. Math.}, 92:761--770, 1970.

 \bibitem{onishchik}
 A. L. Onishchik, E.B. Vinberg.
 \newblock Lie Groups and Lie Algebras III. Structure of Lie Groups and Lie Algebras.
 \newblock {Encyclopaedia of Mathematical Sciences. Springer-Verlag, Berlin}, {\bf 41},  vi+248,
 1994.

\bibitem{plante1}
 J. Plante.
 \newblock Anosov flows.
 \newblock {\em Amer. J. Math. } 94:729--754, (1972)

\bibitem{plante2}
 J. Plante.
 \newblock Anosov flows, transversely affine foliations, and a conjecture of Verjovsky.
 \newblock {\em J. London Math. Soc.} (2)23, no. 2, 359--362, (1981).


\bibitem{push}
C.~Pugh and M.~Shub.
\newblock Ergodicity of {A}nosov actions.
\newblock {\em Invent. Math.}, 15:1--23, 1972.

\bibitem{raghunathan}
M. S. Raghunathan.
\newblock Discrete Subgroups of Lie Groups.
\newblock {\em Springer-Verlag Berlin Heidelberg New York}
ix+227, 1972.


\bibitem{schwartzman}
S. Schwartzman.
\newblock Asymptotic cycles.
\newblock {\em Ann. of Math.,} {\bf 66} (1957), 270--284.

\bibitem{tavarez}
M. Tavarez.
\newblock Anosov actions of nilpotent Lie groups.
\newblock {\em Topology and its Applications.,} {\bf 158} (2011), 636--641.

\bibitem{tomter1} P. Tomter.
   \newblock {A}nosov flows on infra-homogeneous spaces.
\newblock {\em Global Analysis (Proc. Symp. Pure Math., 14)}, Amer. Math. Soc.,
299--327, 1970.


\bibitem{tomter2} P. Tomter.
   \newblock On the classification of {A}nosov flows.
   \newblock {\em Topology} {\bf 14}:179--189, 1975.

\bibitem{vinberg2}
 A. L. Onishchik, E.B. Vinberg.
 \newblock Lie Groups and Lie Algebras II. Discrete subgroups of Lie groups and cohomologies of Lie groups and Lie algebras.
 \newblock {Encyclopaedia of Mathematical Sciences. Springer-Verlag, Berlin}, {\bf 21},  vi+223,
 1994.

\bibitem{verjovsky}
 A. Verjovsky.
     \newblock Codimension one {A}nosov flows.
   \newblock {\em Bol. Soc. Mat. Mexicana}, 19(2):49--77, 1974.


\bibitem{wall}
T. S. Wu.
\newblock A note on  a theorem  on lattices  in Lie  groups.
\newblock {\em  Canad. Math. Bull.}  {\bf 31}  (1988),  no. 2, 190--193.

\end{thebibliography}
\end{document}